\documentclass[12pt,a4paper,reqno]{article}
\usepackage{amsfonts}
\usepackage{amsmath}
\usepackage{amssymb}
\usepackage{amscd}
\usepackage[latin2]{inputenc}
\usepackage{t1enc}
\usepackage[mathscr]{eucal}
\usepackage{indentfirst}
\usepackage{graphicx}
\usepackage{graphics}
\usepackage{pict2e}
\usepackage{epic}
\numberwithin{equation}{section}
\usepackage[margin=1.5cm]{geometry}
\usepackage{epstopdf} 
\usepackage[dvipsnames]{xcolor}
\usepackage{enumerate}
\usepackage{stmaryrd}

\usepackage{mathtools}

\usepackage[LGR,T1]{fontenc}
\DeclareSymbolFont{upgreek}{LGR}{cmr}{m}{n}
\DeclareMathSymbol{\deltaup}{\mathord}{upgreek}{`d}
\DeclareMathSymbol{\piup}{\mathord}{upgreek}{`p}
\DeclareMathSymbol{\epsilonup}{\mathord}{upgreek}{`e}

\newcommand{\huelle}{\operatorname{span}}
\newcommand{\soc}{\operatorname{soc}}

\DeclareMathSymbol{\omegaup}{\mathord}{upgreek}{`w}

\makeatletter
\newcommand*{\shifttext}[2]{%
	\settowidth{\@tempdima}{#2}%
	\makebox[\@tempdima]{\hspace*{#1}#2}%
}
\makeatother

\usepackage{amssymb,amsfonts,amsmath,bbm,mathrsfs,stmaryrd}
\usepackage{soul,xcolor}
\usepackage{url}
\usepackage{ytableau}

\usepackage[colorlinks,
linkcolor=black!75!red,
citecolor=blue,
pdftitle={Root-Reductive},
pdfauthor={Thanasin Nampaisarn},
pdfproducer={pdfLaTeX},
pdfpagemode=None,
bookmarksopen=true
bookmarksnumbered=true]{hyperref}

\usepackage{xr}

\usepackage{tikz,tikz-cd}
\usetikzlibrary{matrix,arrows,calc,decorations.pathreplacing,decorations.markings,intersections,shapes.geometric,through,fit,shapes.symbols,positioning,decorations.pathmorphing,snakes}
\everymath{\displaystyle}

\usepackage[amsmath,thmmarks,hyperref]{ntheorem}
\usepackage{cleveref}

\crefname{section}{Section}{Sections}
\crefformat{section}{#2Section~#1#3} 
\Crefformat{section}{#2Section~#1#3} 

\crefname{subsection}{\S}{\S\S}
\crefformat{subsection}{#2\S#1#3} 
\Crefformat{subsection}{#2\S#1#3}

\makeatletter
\renewtheoremstyle{plain}
{\item[\hskip\labelsep \theorem@headerfont ##1\ ##2\theorem@separator]}%
{\item[\hskip\labelsep \theorem@headerfont ##1\ ##2\ {\normalfont(##3)}\theorem@separator]}
\makeatother
\theoremstyle{plain}

\newtheorem{lem}{Lemma}[section]
\newtheorem{prop}[lem]{Proposition}

\newtheorem{cor}[lem]{Corollary}
\newtheorem{thm}[lem]{Theorem}

\newtheorem{conj}[lem]{Conjecture}

\newtheorem{openq}[lem]{Open Question}

\newtheorem{qf'bis}{\Cref{le.quot_filt'} bis}

\theoremstyle{plain}
\theorembodyfont{\upshape}
\theoremsymbol{\ensuremath{\blacklozenge}}

\newtheorem{define}[lem]{Definition}

\newtheorem{exm}[lem]{Example}
\newtheorem{rmk}[lem]{Remark}

\crefname{definition}{definition}{definitions}
\crefformat{definition}{#2definition~#1#3} 
\Crefformat{definition}{#2Definition~#1#3} 

\crefname{ex}{example}{examples}
\crefformat{example}{#2example~#1#3} 
\Crefformat{example}{#2Example~#1#3} 

\crefname{remark}{remark}{remarks}
\crefformat{remark}{#2remark~#1#3} 
\Crefformat{remark}{#2Remark~#1#3} 

\crefname{convention}{convention}{conventions}
\crefformat{convention}{#2convention~#1#3} 
\Crefformat{convention}{#2Convention~#1#3}

\crefname{lemma}{lemma}{lemmas}
\crefformat{lemma}{#2lemma~#1#3} 
\Crefformat{lemma}{#2Lemma~#1#3} 

\crefname{proposition}{proposition}{propositions}
\crefformat{proposition}{#2proposition~#1#3} 
\Crefformat{proposition}{#2Proposition~#1#3} 

\crefname{corollary}{corollary}{corollaries}
\crefformat{corollary}{#2corollary~#1#3} 
\Crefformat{corollary}{#2Corollary~#1#3} 

\crefname{theorem}{theorem}{theorems}
\crefformat{theorem}{#2theorem~#1#3} 
\Crefformat{theorem}{#2Theorem~#1#3} 

\crefname{assumption}{assumption}{Assumptions}
\crefformat{assumption}{#2assumption~#1#3} 
\Crefformat{assumption}{#2Assumption~#1#3} 

\crefname{equation}{}{}
\crefformat{equation}{(#2#1#3)} 
\Crefformat{equation}{(#2#1#3)}

\theoremstyle{nonumberplain}
\theoremsymbol{\ensuremath{\blacksquare}}

\newtheorem{proof}{Proof}

\newtheorem{proof-of-uptri}{Proof of \Cref{pr.up-tri}}
\newtheorem{proof-of-simples}{Proof of \Cref{pr.simples}}
\newtheorem{proof-of-locnoe}{Proof of \Cref{pr.loc-noe}}
\newtheorem{proof-of-exts}{Proof of \Cref{th.exts}}
\newtheorem{proof-of-univ}{Proof of \Cref{th.univ}}
\newtheorem{proof-of-univ-gen}{Proof of \Cref{th.univ-gen}}

\usepackage{graphics}
\usepackage{epsfig}
\usepackage{epstopdf}
\usepackage{slashed}
\usepackage{enumerate}
\usepackage{scalerel}
\usepackage{epic}
\usepackage{indentfirst}
\usepackage{verbatim}

\usepackage{etoolbox}

\usepackage[T1]{fontenc}
\usepackage{calligra}
\usepackage{xcolor}
\makeatletter
\let\ps@plainorig\ps@plain
\makeatother

\newcommand\swabfamily{\usefont{U}{yswab}{m}{n}}
\DeclareTextFontCommand{\textswab}{\swabfamily}
\newcommand\frakfamily{\usefont{U}{yfrak}{m}{n}}
\DeclareTextFontCommand{\textfrak}{\frakfamily}
\newcommand\gothfamily{\usefont{U}{ygoth}{m}{n}}
\DeclareTextFontCommand{\textgoth}{\gothfamily}

\DeclareMathOperator{\Ext}{Ext}

\DeclareMathOperator{\Hom}{Hom}

\DeclareMathOperator{\im}{im}

\DeclareMathOperator{\ch}{ch}
\DeclareMathOperator{\Res}{Res}

\DeclareFontFamily{U}{mathb}{\hyphenchar\font45}
\DeclareFontShape{U}{mathb}{m}{n}{
	<5> <6> <7> <8> <9> <10> gen * mathb
	<10.95> mathb 10 <12> <14.4> <17.28> <20.74> <24.88> mathb12
}{}
\DeclareSymbolFont{mathb}{U}{mathb}{m}{n}
\DeclareMathSymbol{\precneq}{3}{mathb}{"AC}

\DeclareSymbolFont{euletters}{U}{eur}{m}{n}
\SetSymbolFont{euletters}{bold}{U}{eur}{b}{n}
\DeclareMathSymbol{\varp}{\mathalpha}{euletters}{"7D}

\newcommand{\bggO}{\pazocal{O}}
\newcommand{\bbggO}{{\bar{\pazocal{O}}}}

\newcommand{\defeq}{\overset{\mathrm{def}}{=\!=}}

\newcommand{\amsbb}[1]{\mathbb{#1}}

\newcommand{\gfrak}{\mathfrak{g}}
\newcommand{\hfrak}{\mathfrak{h}}

\newcommand{\bfrak}{\mathfrak{b}}
\newcommand{\nfrak}{\mathfrak{n}}
\newcommand{\pfrak}{\mathfrak{p}}

\newcommand{\Ulie}{\textup{\swabfamily U}}

\newcommand{\Llie}{\textup{\swabfamily L}}
\newcommand{\Plie}{\textup{\swabfamily P}}

\newcommand{\Ilie}{\textup{\swabfamily I}}

\newcommand{\pr}{\textup{pr}}

\usepackage{calrsfs}
\DeclareMathAlphabet{\pazocal}{OMS}{zplm}{m}{n}

\title{Categories $\bggO$ for Root-Reductive Lie Algebras: II. Translation Functors and Tilting Modules}

\author{Thanasin Nampaisarn}

\begin{document}

	\setstcolor{red}
	
	\date{}

	\newcommand{\Addresses}{{
			\bigskip
			\footnotesize
			
			\textsc{Ben-Gurion University of the Negev, Department of Mathematics}\par\nopagebreak
			\textit{E-mail address}: \texttt{namphais@post.bgu.ac.il}
	}}

	\maketitle

	\begin{abstract}
		This is the second paper of a series of papers on a version of categories $\bggO$ for root-reductive Lie algebras.  Let $\gfrak$ be a root-reductive Lie algebra over an algebraically closed field $\amsbb{K}$ of characteristic $0$ with a splitting Borel subalgebra $\bfrak$ containing a splitting maximal toral subalgebra $\hfrak$.  For some pairs of blocks $\bbggO[\lambda]$ and $\bbggO[\mu]$, the subcategories whose objects have finite length are equivalence via functors obtained by the direct limits of translation functors.  Tilting objects can also be defined in $\bbggO$.  There are also universal tilting objects $D(\lambda)$ in parallel to the finite-dimensional cases.
	\end{abstract}
	
	\noindent {\em Key words: root-reductive Lie algebras, finitary Lie algebras, highest-weight modules, BGG categories $\bggO$, equivalences of categories, translation functors, tilting modules}
	
	\vspace{.5cm}

	\noindent{MSC 2010: 17B10; 17B20; 17B22}


	\section*{Introduction}

	The purpose of this paper is to further study a version of Bernstein-Gel'fand-Gel'fand (BGG) categories $\bggO$ for root-reductive Lie algebras with respect to Dynkin Borel subalgebras as defined in \cite{DVN1}.  For a reductive Lie algebra $\gfrak$ over an algebraically closed field of characteristic $0$ whose derived algebra $[\gfrak,\gfrak]$ is finite-dimensional, if $\bggO^\gfrak_\bfrak$ denotes the BGG category $\bggO$ of $\gfrak$ with respect to a certain Borel subalgebra $\bfrak$ of $\gfrak$ as defined in \cite[Chapter 1.1]{bggo}, then we know that some blocks of $\bggO^\gfrak_\bfrak$ are equivalent via translation functors (see \cite[Chapter 1.13]{bggo} and \cite[Chapter 7]{bggo}).  
	
	We denote by $\bggO^\gfrak_\bfrak[\lambda]$ for the block of $\bggO$ containing the simple object $\Llie(\lambda)$ with highest weight $\lambda\in\hfrak^*$, where $\hfrak$ is the Cartan subalgebra of $\gfrak$ contained in $\bfrak$.  Also, $W_{\gfrak,\bfrak}[\lambda]$ is the integral Weyl group for the weight $\lambda$ (see \cite[Chapter 3.4]{bggo}, wherein the notation $W_{[\lambda]}$ is used).  The Borel subalgebra $\bfrak$ induces the set of simple reflections $S_{\gfrak,\bfrak}[\lambda]$, so that $\big(W_{\gfrak,\bfrak}[\lambda],S_{\gfrak,\bfrak}[\lambda]\big)$ is a Coxeter system.
	
	In fact, \cite[Theorem 11]{Soergel} provides a stronger statement, as it a description of the categorical structure of $\bggO^\gfrak_\bfrak$ using the Weyl group of $\gfrak$.  In other words, suppose that $\gfrak$ and $\gfrak'$ are two reductive Lie algebras with Borel subalgebras $\bfrak$ and $\bfrak'$; for $\lambda\in\hfrak^*$ and $\lambda'\in(\hfrak')^*$, if the Coxeter systems $\big(W_{\gfrak,\bfrak}[\lambda],S_{\gfrak,\bfrak}[\lambda]\big)$ and $\big(W_{\gfrak',\bfrak'}[\lambda'],S_{\gfrak',\bfrak'}[\lambda']\big)$ are isomorphic, then the blocks $\bggO^{\gfrak}_{\bfrak}[\lambda]$ and $\bggO^{\gfrak'}_{\bfrak'}[\lambda']$ are equivalent as categories.
	
	The paper \cite{fiebig} studies Kac-Moody algebras and obtains a similar result to \cite[Theorem 11]{Soergel}.  If $\gfrak$ and $\gfrak'$ are complex symmetrizable Kac-Moody algebras with Borel subalgebras $\bfrak$ and $\bfrak'$ and Cartan subalgebras $\hfrak$ and $\hfrak'$, where $\hfrak\subseteq \bfrak\subseteq\gfrak$ and $\hfrak'\subseteq\bfrak'\subseteq \gfrak'$.  We denote by $\bggO^{\gfrak}_{\bfrak}$ and $\bggO^{\gfrak'}_{\bfrak'}$ for the corresponding BGG categories $\bggO$ for the pairs $(\gfrak,\bfrak)$ and $(\gfrak',\bfrak')$, respectively.  Let $\Lambda\subseteq \hfrak^*$ be the set of highest weights of simple objects in a block of $\bggO^{\gfrak}_{\bfrak}$, and write $\bggO_{\Lambda}$ for the said block.  The notations $\Lambda'\subseteq (\hfrak')^*$ and $\bggO'_{\Lambda'}$ are defined similarly for $\bggO^{\gfrak'}_{\bfrak'}$.  For specific pairs $\Lambda$ and $\Lambda'$, \cite[Theorem 4.1]{fiebig} establishes an equivalence between the categories $\bggO_\Lambda$ and $\bggO'_{\Lambda'}$.  One of the necessary conditions for the existence of an equivalence in \cite[Theorem 4.1]{fiebig} is that there exists an isomorphism between relevant Coxeter systems.
	
	In the present paper, we shall look at the subcategory $\mathcal{O}^\gfrak_\bfrak[\lambda]$ consisting of objects of finite length from the block $\bbggO^\gfrak_\bfrak[\lambda]$, where $\bbggO$ is an extended BGG category $\bggO$ for a root-reductive Lie algebra $\gfrak$ with respect to a Dynkin Borel subalgebras $\bfrak$,  and $\lambda\in\hfrak^*$.  Here, $\hfrak$ is the unique splitting maximal toral subalgebra of $\gfrak$ contained in $\bfrak$.  We obtain a similar result to \cite[Theorem 11]{Soergel} and \cite[Theorem 4.1]{fiebig}, with \cite[Theorem 11]{Soergel} being the crucial ingredient for our proof.
	
	Another main topic of this paper is tilting theory.  In the case where $\gfrak$ is a reductive Lie algebra with $[\gfrak,\gfrak]$ being finite-dimensional, tilting modules in $\bggO^\gfrak_\bfrak$ are objects with both standard and costandard filtrations (see \cite[Chapter 11.1]{bggo}).  For a root-reductive Lie algebra $\gfrak$, indecomposable objects in $\bbggO^\gfrak_\bfrak$ can potentially have infinite length, in which case the notion of filtrations may not apply to $\bbggO$.   However, if we generalize the definition of filtrations, then it is possible to define tilting modules in $\bbggO^\gfrak_\bfrak$ in a similar manner.  	
	
	
	This paper consists of three sections.  The first section provides necessary foundations for other sections such as a brief recapitulation of the results from \cite{DVN1} and some relevant definitions such as generalized filtrations.  This section also provides a characterization of integrable modules in our version of BGG categories $\bggO$ for root-reductive Lie algebras, which are usually denoted by $\bbggO$.  The second section provides a visualization of the subcategory of each block of $\bbggO$ consisting of modules of finite length, proving that each subcategory is a direct limit of subcategories of some categories $\bggO$ for reductive Lie algebras with finite-dimensional derived algebras. The final section deals with the construction and the properties of tilting modules in $\bbggO$.

\subsection*{Acknowledgement}

	The author has been supported by ISF Grant 711/18.  The author would also like to thank Dr. Inna Entova from Ben-Gurion University of the Negev and Professor Maria Gorelik from Weizmann Institute of Science for the help towards the completion of this paper.






	\section{Preliminaries}
	
	All vector spaces and Lie algebras are defined over an algebraically closed field $\mathbb{K}$ of characteristic $0$.  For a vector space $V$, $\dim V$ is the $\mathbb{K}$-dimension of $V$ and $V^*$ denotes its algebraic dual $\Hom_\mathbb{K}(V,\mathbb{K})$.  Unless otherwise specified, the tensor product $\otimes$ is defined over $\mathbb{K}$.  For a Lie algebra $\gfrak$, $\mathfrak{U}(\gfrak)$ is its universal enveloping algebra.


	\subsection{Root-Reductive Lie Algebras and Categories $\bbggO$}
	
	Let $\mathfrak{g}$ be a root-reductive Lie algebra in the sense of \cite[Definitioin 1.1]{DVN1}.  Suppose that $\mathfrak{h}$ is a splitting maximal toral subalgebra of $\gfrak$ in the sense of \cite[Definitioin 1.2]{DVN1}, and $\mathfrak{b}$ is a Dynkin Borel subalgebra of $\mathfrak{g}$ (see \cite[Definition 1.5]{DVN1}) that contains $\mathfrak{h}$.  Let $\nfrak:=[\bfrak,\bfrak]$ (we sometimes write $\bfrak^+$ and $\nfrak^+$ for $\bfrak$ and $\nfrak$, respectively).  If $\bfrak^-$ is the unique Borel subalgebra of $\gfrak$ such that $\bfrak^+\cap\bfrak^-=\hfrak$, then we have the following decompositions of vector spaces: $\bfrak^{\pm}=\hfrak\oplus\nfrak^{\pm}$ and $\gfrak=\nfrak^-\oplus\hfrak\oplus\nfrak^+$.
	
	For each $\mathfrak{h}$-root $\alpha$ of $\mathfrak{g}$, the $\hfrak$-root space of $\mathfrak{g}$ associated to $\alpha$ is given by $\gfrak^\alpha$.   With respect to $\mathfrak{b}$, the set $\Phi_{\mathfrak{g},\mathfrak{h}}$ of $\mathfrak{h}$-roots of $\mathfrak{g}$ can be partitioned into two disjoint subsets $\Phi^+_{\mathfrak{g},\mathfrak{b}}$ consisting of positive $\mathfrak{b}$-roots and $\Phi^-_{\mathfrak{g},\mathfrak{b}}$ consisting of negative $\mathfrak{b}$-roots.  Write $W_{\gfrak,\hfrak}$ for the Weyl group of $\Phi_{\gfrak,\hfrak}$.  Also, $\Lambda_{\gfrak,\hfrak}:=\operatorname{span}_{\mathbb{Z}}\Phi_{\gfrak,\hfrak}$.  When there is no risk of confusion, we shall write $\Phi$, $\Phi^+$, $\Phi^-$, $W$ , and $\Lambda$ for $\Phi_{\mathfrak{g},\mathfrak{h}}$, $\Phi^+_{\mathfrak{g},\mathfrak{b}}$, $\Phi^-_{\mathfrak{g},\mathfrak{b}}$, $W_{\gfrak,\hfrak}$, and $\Lambda_{\gfrak,\hfrak}$, respectively.  
	
	The set of (positive) $\bfrak$-simple roots is denoted by $\Sigma_{\gfrak,\bfrak}$, or $\Sigma^+_{\gfrak,\bfrak}$.  The set of negative $\bfrak$-simple roots is given by $\Sigma^-_{\gfrak,\bfrak}$.  For convenience, we also write $\Sigma$, $\Sigma^+$, and $\Sigma^-$ for $\Sigma_{\gfrak,\bfrak}$, $\Sigma^+_{\gfrak,\bfrak}$, and $\Sigma^-_{\gfrak,\bfrak}$.  For each $\alpha\in \Phi$, let $x_{+\alpha}\in \gfrak^{+\alpha}$, $x_{-\alpha}\in\gfrak^{-\alpha}$, and $h_\alpha\in \left[\gfrak^{+\alpha},\gfrak^{-\alpha}\right]$ be such that $h_\alpha=\left[x_{+\alpha},x_{-\alpha}\right]$ and $\alpha\left(h_\alpha\right)=2$ (that is, $h_\alpha$ is the coroot of $\alpha$).  Thus, $\big\{x_\alpha\,\big|\,\alpha\in \Phi\big\}\cup\big\{h_\alpha\,\big|\,\alpha\in\Sigma\big\}$ is a Chevalley basis of $[\gfrak,\gfrak]$.
	
	For convenience, we fix a filtration $\dot{\gfrak}_1\subseteq \dot{\gfrak}_2\subseteq \dot{\gfrak}_3\subseteq\ldots$ of $\mathfrak{g}$ such that each $\dot{\gfrak}_n$ is a finite-dimensional reductive Lie algebra with $\dot{\bfrak}_n:=\bfrak\cap\dot{\gfrak}_n$ and $\dot{\hfrak}_n:=\hfrak\cap\dot{\gfrak}_n$ as a Borel subalgebra and a Cartan subalgebra, respectively.  We also define $\gfrak_n:=\dot{\gfrak}_n+\hfrak$ and $\bfrak_n:=\dot{\bfrak}_n+\hfrak$.  The notations $\dot{\bfrak}_n^\pm$, $\bfrak^\pm_n$, $\nfrak^{\pm}_n$, and $\nfrak_n$ carry similar meanings.  
	
	Note that there exists $\rho\in\hfrak^*$ such that $\rho|_{\dot{\hfrak}_n}$ is the half sum of $\dot{\bfrak}_n$-positive roots of each $\dot{\gfrak}_n$.  Then, we define the dot action of $W$ on $\hfrak^*$ by $w\cdot\lambda=w(\lambda+\rho)-\rho$ for each $\lambda\in\hfrak^*$.  While the map $\rho$ may not be unique, the dot action is independent of the choice of $\rho$.  For a fixed $\lambda\in\hfrak^*$, the subgroup $W_{\gfrak,\bfrak}[\lambda]$ (also denoted by $W[\lambda]$) of $W_{\gfrak,\hfrak}$ consists of $w\in W_{\gfrak,\hfrak}$ such that $w\cdot\lambda-\lambda\in\Lambda_{\gfrak,\hfrak}$.
	
	Write $\bbggO^{\mathfrak{g}}_{\bfrak}$ for the extended category $\bggO$ for the pair $(\mathfrak{g},\mathfrak{b})$ (see \cite[Definition 2.1]{DVN1}).  For simplicity, we also write $\bbggO$ for $\bbggO^{\gfrak}_{\hfrak}$.  For each $M\in \bbggO$ and $\lambda\in\hfrak^*$, $M^\lambda$ denotes the $\mathfrak{h}$-weight space with respect to the weight $\lambda$.  Let $(\_)^\vee: \bbggO^{\mathfrak{g}}_{\bfrak}\rightsquigarrow  \bbggO^{\mathfrak{g}}_{\bfrak}$ be the duality functor as defined in \cite[Definition 2.2.]{DVN1}.  Then, $\Delta^{\gfrak}_{\bfrak}(\lambda)$, or simply $\Delta(\lambda)$, is defined to be the Verma module with highest weight $\lambda\in\hfrak^*$ (see \cite[Definition 1.9]{DVN1}).  We also write $\nabla^\gfrak_\bfrak(\lambda)$, or simply $\nabla(\lambda)$, for the co-Verma module with highest weight $\lambda\in\hfrak^*$, namely, $\nabla(\lambda)=\big(\Delta(\lambda)\big)^\vee$.  Write $\Llie^\gfrak_\bfrak(\lambda)$, or simply $\Llie(\lambda)$, for the simple quotient of $\Delta(\lambda)$. 
	
	For each $\lambda\in \hfrak^*$, we shall denote by $[\lambda]$ the set of all weights $\mu\in W\cdot \lambda$ such that $\lambda-\mu\in \Lambda$.  The definition of \emph{abstract blocks} is given by \cite[Definition 4.13]{Penkov-Chirvasitu}.  A \emph{block} in our consideration is the full subcategory of $\bbggO$ consisting of all objects belonging in the same abstract block.  Due to \cite[Theorem 3.4]{DVN1}, we can write each block of $\bbggO$ as $\bbggO[\lambda]$, where $\bbggO[\lambda]$ is the unique block of $\bbggO$ that contains $\Delta(\lambda)$.  If $\Omega^\gfrak_\bfrak$, or simply $\Omega$, is the set of all $[\lambda]$, where $\lambda\in\hfrak^*$, then
	\begin{align}
		\bbggO=\bigoplus_{[\lambda]\in\Omega}\,\bbggO[\lambda]\,.
	\end{align}
	Note that, if $\gfrak$ is finite-dimensional, then $\bbggO[\lambda]=\bggO[\lambda]$.
	
	For a given $\lambda\in\hfrak^*$, let $\pr_{\gfrak,\bfrak}^\lambda:\bbggO\rightsquigarrow \bbggO[\lambda]$ denote the projection onto the $[\lambda]$-block.  We also write $\operatorname{inj}_{\gfrak,\bfrak}^\lambda:\bbggO[\lambda]\rightsquigarrow\bbggO$ for the injection from the $[\lambda]$-block.  When the context is clear, $\pr^\lambda$ and $\operatorname{inj}^\lambda$ are used instead.  Note that both functors are exact, and are adjoint to one another.
	
	 For each $M\in \bbggO$, $\Pi(M)\defeq\big\{\lambda\in\hfrak^*\,\big|\,M^\lambda\ne0\big\}$ and $\operatorname{ch}_\hfrak(M)\defeq\sum\limits_{\lambda\in\hfrak^*}\,\dim (M^\lambda)\,e^\lambda$ is the formal character of $M$ (with the standard multiplication rule given by $e^\lambda e^\mu\defeq e^{\lambda+\mu}$ for all $\lambda,\mu\in\hfrak^*$).  When the context is clear, $\ch(M)$ denotes $\ch_\mathfrak{h}(M)$.  We note that $\operatorname{ch}(M)=\sum\limits_{\lambda\in\hfrak^*}\,\big[M:\Llie(\lambda)\big]\,\operatorname{ch}\big(\Llie(\lambda)\big)$, where $[M:L]$ denotes the multiplicity of a simple object $L$ in $M$ (see \cite[Corollary 2.10]{DVN1}).  If $L$ is a simple object such that $[M:L]>0$, then $L$ is called a \emph{composition factor} of $M$.

	
	\subsection{Generalized Filtrations}
	
	Let $\mathcal{C}$ be an abelian category.  Fix a family $\mathcal{F}\subseteq\mathcal{C}$.
	
	\begin{define}
		A \emph{generalized $\mathcal{F}$-filtration} of $M\in\mathcal{C}$ is a collection $(M_j)_{j\in J}$ of subobjects $M_j$ of $M$, where $(J,\preceq)$ is a totally ordered set, such that
		\begin{enumerate}[(i)]
			\item $M_j\subsetneq M_k$ for all $j,k\in J$ such that $j\prec k$,
			\item $\bigcap_{j\in J}\,M_j=0$,
			\item $\bigcup_{j\in J}\,M_j=M$, and
			\item for each $j\in J$, $M_j\Big/\left(\bigcup\limits_{k\prec j}\,M_k\right)$ is an object in $\mathcal{F}$.
		\end{enumerate}
	\end{define}
	
	\begin{define}
		Two generalized $\mathcal{F}$-filtrations $(M_j)_{j\in J}$ and $(M'_{j'})_{j'\in J'}$ of $M\in\mathcal{C}$ (where $\preceq$ and $\preceq'$ are the respective total orders on $J$ and $J'$) are said to be \emph{$\mathcal{F}$-equivalent} if there exists a bijectioin $f:J\to J'$ such that $M_j\Big/\left(\bigcup\limits_{k\prec j}\,M_k\right) \cong M'_{f(j)}\Big/\left(\bigcup\limits_{k'\prec f(j)}\,M'_{k'}\right)$ for every $j\in J$ that is not the least element of $J$.
	\end{define}
	
	\begin{define}
	 	We say that $\mathcal{F}$ is a \emph{complete filter} if, for any $M\in\mathcal{C}$, $M$ has a generalized $\mathcal{F}$-filtration.  We say that $\mathcal{F}$ is a \emph{good filter} if any two filtrations $(M_j)_{j\in J}$ and $(M'_{j'})_{j'\in J'}$ of a single object $M\in\mathcal{C}$ are $\mathcal{F}$-equivalent.
	\end{define}

	\begin{define}
	 	We define $\mathcal{F}(\mathcal{C})$ to be the full subcategory of $\mathcal{C}$ whose objects are those with generalized $\mathcal{F}$-filtrations.
	\end{define}
	
	Consider $\mathcal{C}:=\bbggO$.  We have the following theorem.
	
	\begin{thm}
	Define $\boldsymbol{\Delta}:=\big\{\Delta(\lambda)\,\big|\,\lambda\in\hfrak^*\big\}$, $\boldsymbol{\nabla}:=\big\{\nabla(\lambda)\,\big|\,\lambda\in\hfrak^*\big\}$, and  $\boldsymbol{L}:=\big\{\Llie(\lambda)\,\big|\,\lambda\in\hfrak^*\big\}$ as subcollections of the category $\bbggO$.
	\begin{enumerate}[(a)]
		\item The collection $\boldsymbol{\Delta}$ is a good filter of $\bbggO$.  (A generalized $\boldsymbol{\Delta}$-filtration is also called a generalized standard filtration.)
		\item The collection $\boldsymbol{\nabla}$ is a good filter of $\bbggO$.  (A generalized $\boldsymbol{\nabla}$-filtration is also called a generalized standard filtration.)
		\item The collection $\boldsymbol{L}$  is a good and complete filter of $\bbggO$.  (A generalized $\boldsymbol{L}$-filtration is also known as a generalized composition series.)
	\end{enumerate}
	\label{thm:goodfilter}
	\end{thm}
	
	\begin{proof} Part (c) follows from \cite[Corollary 2.10]{DVN1}.  By employing duality, Part (b) is a trivial consequence of Part (a).  We shall now prove Part (a).
		
		Suppose that $M\in\bbggO$ has two generalized standard filtrations $(M_j)_{j\in J}$ and $(M'_{j'})_{j'\in J'}$ of $M\in\mathcal{C}$ (where $\preceq$ and $\preceq'$ are the respective total orders on $J$ and $J'$).  Without loss of generality, we may assume that $M$ lies a single block $\bbggO[\lambda]$ of $\bbggO$.
	
		Let $p$ denote the formal character of $\Delta(0)$.  For $j\in J$ and $j'\in J'$ that are not the least elements of $J$ and $J'$, respectively, suppose that $\mu(j)$ and $\mu'(j')$ denote the highest weights of $M_j\Big/\bigcup_{k\prec j}\,M_k$ and $M'_{j'}\Big/\bigcup_{k'\prec'j'}\,M'_{k'}$, respectively.  Write $a$ and $a'$ for the sum of all $e^{\mu(j)}$ and the sum of all $e^{\mu'(j')}$, respectively.  It follows that $\ch(M)=ap$ and $\ch(M)=a'p$.
		
		We now let $q$ to be the infinite product of $e^0-e^{-\alpha}$, where $\alpha$ runs over all $\bfrak$-positive roots.  Since $\bfrak$ is a Dynkin Borel subalgebra, $q$ is well defined.  We can easily show that $pq=e^0$.  Therefore, 
		\begin{align}
			a=ae^0=a(pq)=(ap)q=\ch(M)\,q=(a'p)q=a'(pq)=a'e^0=a'\,.
		\end{align}
		The claim follows immediately.
	\end{proof}
	
	\begin{cor}
		For each object $M\in\boldsymbol{\Delta}(\bbggO)$ and a given generalized standard filtration $(M_j)_{j\in J}$ of $M$, the number of times $\Delta(\lambda)$ occurs as a quotient $M_j\Big/\bigcup_{k\prec j}\,M_k$ is a finite nonnegative integer, which is independent of the choice of the generalized standard filtration $(M_j)_{j\in J}$.  This number is denoted by $\big\{M:\Delta(\lambda)\big\}$.
		
		For each $M\in\boldsymbol{\nabla}(\bbggO)$ and a given generalized co-standard filtration $(M_j)_{j\in J}$ of $M$, the number of times $\nabla(\lambda)$ occurs as a quotient $M_j\Big/\bigcup_{k\prec j}\,M_k$ is a finite nonnegative integer which is independent of the choice of the generalized co-standard filtration.  This number is denoted by $\big\{M:\nabla(\lambda)\big\}$.
	\end{cor}
	
	\begin{exm}
		For fixed $\lambda,\mu\in\hfrak^*$ such that $\mu\preceq\lambda$, the truncated projective cover $\Plie^{\preceq\lambda}(\mu)$ of $\Llie(\lambda)$ lies in $\boldsymbol{\Delta}(\bbggO)$, whilst the truncated injective hull $\Ilie^{\preceq\lambda}(\mu)$of $\Llie(\lambda)$ lies in $\boldsymbol{\nabla}(\bbggO)$.  See \cite[Section 4]{DVN1}.
	\end{exm}
	
	\begin{prop}
		Suppose that $M\in\boldsymbol{\Delta}(\bbggO)$.  
		\begin{enumerate}[(a)]
			\item  If $\lambda$ is a maximal weight of $M$, then $M$ has a submodule $N$ isomorphic to $\Delta(\lambda)$, and the factor module $M/N$ is in $\boldsymbol{\Delta}(\bbggO)$.
			\item  If $N$ is a direct summand of $M$, then $N\in\boldsymbol{\Delta}(\bbggO)$.
			\item  The module $M$ is a free $\Ulie(\nfrak^-)$-module.
		\end{enumerate}
		\label{prop:standardfiltration}
	\end{prop}
	
	\begin{proof}
		For Part (a), let $u$ be a maximal vector of $M$ with weight $\lambda$.   There exists an index $j\in J$ such that $u \in M_j$ but $u\notin M_{\prec j}:=\bigcup_{k\prec j}\,M_k$.  Thus, the map $\psi:\Delta(\lambda)\to M_j/M_{\prec j}$ sending $g\cdot u'\mapsto g\cdot u+M_{\prec j}$ for each $g\in \Ulie(\gfrak)$, where $u'$ is a maximal vector of $\Delta(\lambda)$, is a nonzero homomorphism of Verma modules (recalling that $M_j/M_{\prec j}$ is a Verma module).  By  \cite[Theorem 1.1]{DVN1}, $\psi$ must be injective.  Therefore, $\im(\psi)$ is a Verma submodule with highest weight $\lambda$ of the Verma module $M_j/M_{\prec j}$.  Because $\lambda$ is a maximal weight of $M$, we conclude that $\im(\psi)=M_j/M_{\prec j}$ and $\psi$ is an isomorphism of $\gfrak$-modules.  Hence, the $\gfrak$-submodule $N$ of $M$ generated by $u$ is a Verma module isomorphic to $\Delta(\lambda)$.  We now note that $M/M_j$ and $M_{\prec j}$ are both in $\boldsymbol{\Delta}(\bbggO)$.  Furthermore, because $N\cap M_{\prec j}\cong\ker(\psi)=0$, we obtain a short exact sequence $0\to M_{\prec j}\to M/N \to M/M_j\to 0$.  Hence, $M/N$ has a generalized standard filtration given by patching the generalized standard filtration of $M_{\prec j}$ with the generalized standard filtration of $M/M_j$.
		
		For Part (b), we may assume without loss of generality that $N$ is an indecomposable direct summand of $M$.  Now, define $N[0]:=N$ and $\lambda[0]:=\lambda$.  We finish the proof using transfinite induction.  For an ordinal $t$ with a predecessor $s$, suppose a module $N[s]$ and a weight $\lambda[s]\in\hfrak^*$ are given such that $\Delta\big(\lambda[s]\big)\subseteq N[s]$.  Then, define $N[t]:=N[s]/\Delta\big(\lambda[s]\big)$.  If $N[t]=0$, then we are done.  If $N[t]\neq 0$, then by taking $\lambda[t]$ to be a maximal weight of $N[t]$, using the same idea as the paragraph above, we conclude that $\Delta\big(\lambda[t]\big)\subseteq N[t]$.  On the other hand, if $t$ is a limit ordinal, then we have a directed system of modules $\big(N[s]\big)_{s<t}$.  Define $N[t]$ to be the direct limit of the modules $N[s]$ for $s<t$.  As before, if $N[t]=0$, then we are done.  If not, we then take $\lambda[t]$ to be a maximal weight of $N[t]$.  Then, again,  $\Delta\big(\lambda[t]\big)\subseteq N[t]$.   Since the multiset of composition factors of $N$ is a countable multiset, this procedure must stop at some countable ordinal $\tau$, where $N[\tau]=0$.  We then obtain a generalized standard filtration of $N$.  
		
		For Part (c), let $\left(M_j\right)_{j\in J}$ be a generalized standard filtration of $M$.  For an element $j\in J$ that is not the minimum element of $J$, take $m_j\in M_j\setminus \bigcup_{k\prec j}\,M_k$ such that $m_j$ is a weight vector whose weight is the highest weight of $M_j\Big/\bigcup_{k\prec j}\,M_k$.  Then, $M$ is a free $\Ulie(\nfrak^-)$-module with basis $\left\{m_j\,|\,j\in J\right\}$.
	\end{proof}
	
	\begin{cor}
		Suppose that $M\in\boldsymbol{\nabla}(\bbggO)$.  
		\begin{enumerate}[(a)]
			\item  If $\lambda$ is a maximal weight of $M$, then $M$ has a submodule $N$ such that $M/N$ is isomorphic to $\nabla(\lambda)$, and the submodule $N$ is in $\boldsymbol{\nabla}(\bbggO)$.
			\item  If $N$ is a direct summand of $M$, then $N\in\boldsymbol{\nabla}(\bbggO)$.
			\item  The module $M$ is a free $\Ulie(\nfrak^-)$-module.
		\end{enumerate}
		\label{cor:costandardfiltration}
	\end{cor}
	
	\begin{define}
		Let $\pazocal{D}^\gfrak_\bfrak$ (or simply, $\pazocal{D}$)  denote the subcategory $\boldsymbol{\Delta}(\bbggO^\gfrak_\bfrak)\cap\boldsymbol{\nabla}(\bbggO^\gfrak_\bfrak)$.  The objects in $\pazocal{D}$ are called \emph{tilting modules}.
	\end{define}


\subsection{Integrable Modules}
	
	\begin{define}
		Let $\mathfrak{a}$ be an arbitrary Lie algebra.  An $\mathfrak{a}$-module $M$ is said to be \emph{integrable} (or \emph{$\mathfrak{a}$-integrable}) if, for any $m\in M$ and $a\in\mathfrak{a}$, the elements $m$, $a\cdot m$, $a^2\cdot m$, $\ldots$ span a finite-dimensional subspace of $M$.
	\end{define}
	
	\begin{define}
		Let $\lambda\in\hfrak^*$.
		\begin{enumerate}[(a)]
			\item We say that $\lambda\in\mathfrak{h}^*$ is \emph{integral (with respect to $\gfrak$ and $\hfrak$)} if $h_\alpha(\lambda)$ is an integer for all $\alpha\in\Phi$.  
			\item If  $h_\alpha(\lambda)\in\mathbb{Z}_{\geq 0}$ for every $\alpha\in \Phi^+$, then $\lambda$ is said to be \emph{dominant-integral (with respect to $\gfrak$ and $\bfrak$)}.  
			\item If $h_\alpha(\lambda)\notin\mathbb{Z}$ for all $\alpha\in \Phi$, then $\lambda$ is said to be \emph{nonintegral (with respect to $\gfrak$ and $\hfrak$)}.  
			\item If $h_\alpha(\lambda)\notin\mathbb{Z}$ for all but finitely many $\alpha\in \Phi$, then $\lambda$ is said to be \emph{almost nonintegral (with respect to $\gfrak$ and $\hfrak$)}.  
		\end{enumerate}
	\end{define}
	
	\begin{thm}
		A module $M\in\bbggO$ is integrable if and only if it is a direct sum of simple integrable modules in $\bbggO$.  All simple integrable modules in $\bbggO$ are of the form $\Llie(\lambda)$, where $\lambda\in\hfrak^*$ is dominant-integral.  Nonisomorphic simple integrable modules belong in different blocks of $\bbggO$.
	\end{thm}
	
	\begin{proof}
		For the first statement, we may assume that $M$ is an indecomposable module (by means of \cite[Corollary 2.6]{DVN1}).  We shall prove that $M$ is simple.  Let $u$ be a highest-weight vector of $M$ associated to the weight $\lambda$.  We claim that $\lambda$ is dominant-integral and $M\cong\Llie(\lambda)$. 
		
		First, by considering $M$ as a $\gfrak_n$-module, we easily see that $M$ is a direct sum of simple finite-dimensional $\gfrak_n$-modules.  In particular, the $\gfrak_n$-submodule $M_n$ generated by $u$ is a simple direct summand of $M$.  Note that $M_1\subseteq M_2\subseteq M_3\subseteq \ldots$.  If $M'$ is the union of $\bigcup\limits_{n\in\mathbb{Z}_{>0}}\,M_n$ where each $M_n$ is a simple $\gfrak_n$-module, then $M'$ is a simple $\gfrak$-module isomorphic to $\Llie(\lambda)$.  Clearly, $\lambda$ must be a dominant-integral weight with respect to the Lie algebra $\gfrak_n$ with the Borel subalgebra $\bfrak_n$.  Therefore, $\lambda$ is dominant-integral with respect to $\gfrak$ and $\bfrak$. 
		
		If $M\neq M'$, then $M/M'$ has a highest-weight vector of the form $u'+M'$, where $v\in M$.  Using the same argument, the submodule $M''$ of $M/M'$ generated by $u'+M'$ is isomorphic to $\Llie(\mu)$ for some dominant-integral weight $\mu\in\hfrak^*$ with respect to $\gfrak$ and $\bfrak$.  Since $M$ is indecomposable, by \cite[Proposition 3.2]{DVN1}, we conclude that $\mu\in[\lambda]$.  However, the only dominant-integral weight in $[\lambda]$ is $\lambda$ itself.  Consequently, $\mu=\lambda$.  This shows that the only possible composition factor of $M$ is $\Llie(\lambda)$.  However, there are no nontrivial extensions of $\Llie(\lambda)$ by itself (see \cite[Proposition 3.8(d)]{DVN1}).
		
		Finally, we shall prove that every simple module of the form $\Llie(\lambda)$ is integrable if $\lambda\in\hfrak^*$ is dominant-integral.  Let $v$ be a highest-weight vector of $L:=\Llie(\lambda)$.  Define $L_n:=\Ulie(\gfrak_n)\cdot v$ for every $n\in\mathbb{Z}_{>0}$.  Clearly, each $L_n\cong \Llie^{\gfrak_n}_{\bfrak_n}(\lambda)$ is finite-dimensional with $L_1\subseteq L_2\subseteq L_3\subseteq\ldots$, and $L=\bigcup\limits_{n\in\mathbb{Z}_{>0}}\,L_n$.  Fix $k\in\mathbb{Z}_{>0}$.  Now, for $n\geq k$, observe that each $L_n$ is a direct sum of simple finite-dimensional $\gfrak_k$-modules.  Consequently,  for each $n\geq k$, $L_{n+1}=L_n\oplus F^k_n$ with $F^k_n$ being a direct sum of simple finite-dimensional $\gfrak_k$-modules.  That is, $L=L_k\oplus\bigoplus \limits_{n\geq k}\,F^k_n$ is a direct sum of simple finite-dimensional $\gfrak_k$-modules.  It follows immediately that $L$ is an integrable $\gfrak$-module.
	\end{proof}
	
	\begin{cor}
		Let $\bbggO_{\textup{integrable}}$ denote the full subcategory of $\bbggO$ consisting of integrable modules.  Then, $\bbggO_{\textup{integrable}}$ is semisimple.
	\end{cor}
	
	The theorem below gives another way to verify whether a module $M\in\bbggO$ is $\gfrak$-integrable.  Recall that $\Pi(M)$ is the set of $\hfrak$-weights of $M$.
	
	\begin{thm}
		Let $M\in\bbggO$.  The following consitions on $M$ are equivalent:
		\begin{enumerate}[(i)]
			\item $M$ is $\gfrak$-integrable;
			\item $M$ is $\nfrak^-$-integrable;
			\item for all $w\in W$ and $\lambda\in\hfrak^*$, $\dim (M^\lambda)=\dim (M^{w\lambda})$;
			\item the set $\Pi(M)$ is stable under the natural action of $W$.
		\end{enumerate}
	\end{thm}
	
	\begin{proof}
		
		The direction (i)$\Rightarrow$(ii) is obvious.  For (ii)$\Rightarrow$(iii), we note that $w\in W_n$ for any sufficiently large positive integer $n$.  Let $M_n^{(\lambda)}$ be the $\mathfrak{g}_n$-submodule of $M$ given by
		\begin{align}
			M_n^{(\lambda)}:=\mathfrak{U}\big(\gfrak_n\big)\cdot M^\lambda\,.
			\label{eq:Mnlambda}
		\end{align}
		Because $M$ is $\mathfrak{n}^-$-integrable, $M^\lambda$ is finite-dimensional, $\mathfrak{U}\big(\mathfrak{g}_n\big)=\mathfrak{U}\big(\mathfrak{n}_n^-\big)\cdot \mathfrak{U}(\mathfrak{h})\cdot\mathfrak{U}\big(\mathfrak{n}_n^+\big)$, and $M$ is locally $\mathfrak{U}(\mathfrak{n}^+)$-finite, we see that $M$ is locally $\mathfrak{U}\big(\mathfrak{n}_n^-\big)$-finite, whence $M_n^{(\lambda)}$ is a finite-dimensional $\mathfrak{g}_n$-submodule of $M$.  For all sufficiently large $n$, we have $\big(M_n^{(\lambda)}\big)^{w \lambda} = M^{w\lambda}$.  Since the support of a finite-dimensional module is invariant under the action of the Weyl group, the claim follows.
		
		The statement (iii)$\Rightarrow$(iv) is trivial.  We now prove (iv)$\Rightarrow$(i).  Let $M_n^{(\lambda)}$ be the module \eqref{eq:Mnlambda}.  Because $M$ is locally $\mathfrak{U}(\mathfrak{n}^+)$-finite, $\Pi(M)$ is invariant under $W_{\gfrak,\hfrak}$, and $M$ has finite-dimensional $\mathfrak{h}$-weight spaces, we conclude that the weights $M_n^{(\lambda)}$ must lie in the orbit of $\lambda$ under $W_{\gfrak_n,\hfrak_n}$, making $M_n^{(\lambda)}$ is a finite-dimensional $\mathfrak{g}_n$-module.  Hence, $M=\bigcup_{n\in\mathbb{Z}_{>0}}\,\bigcup_{\lambda\in\Pi(M)}\,M_n^{(\lambda)}$ is an integrable $\gfrak$-module.
	\end{proof}








\section{Translation Functors}


	\subsection{Some Functors between Extended Categories $\bggO$}
	
	Fix a positive integer $n$.  Let $M_n\in\bbggO^{\gfrak_n}_{\bfrak_n}$.  Define $\pfrak_{n+1}$ to be the parabolic subalgebra $\gfrak_n+\bfrak_{n+1}$ of $\gfrak_{n+1}$.  It can be easily seen that $M_n$ is a $\pfrak_{n+1}$-module, where $\bfrak_{n+1}$ acts on $M_n$ via $x_\alpha\cdot m=0$ for all $m\in M_n$ and $\alpha\in\Phi_{\gfrak_{n+1},\bfrak_{n+1}}^+\setminus\Phi_{\gfrak_n,\bfrak_n}$.

	\begin{define}
		Define $I_{n+1}:\bbggO^{\gfrak_n}_{\bfrak_n}\rightsquigarrow\bbggO_{\bfrak_{n+1}}^{\gfrak_{n+1}}$ to be the parabolic induction functor \begin{align}I_{n+1}M_n:=\Ulie(\gfrak_{n+1})\underset{\Ulie(\pfrak_{n+1})}{\otimes}M_n\end{align} for all $M_n\in\bbggO^{\gfrak_n}_{\bfrak_n}$.
	\end{define}
	
	\begin{prop}
		The functor $I_{n+1}:\bbggO^{\gfrak_n}_{\bfrak_n}\rightsquigarrow\bbggO_{\bfrak_{n+1}}^{\gfrak_{n+1}}$ is exact.
	\end{prop}
	
	\begin{proof}
		The $\Ulie(\pfrak_{n+1})$-module $\Ulie(\gfrak_{n+1})$ is a free module due to the Poincar\'e-Birkhoff-Witt (PBW) Theorem.  Thus, $\Ulie(\gfrak_{n+1})$ is a flat $\Ulie(\pfrak_{n+1})$-module.
	\end{proof}
	
	\begin{rmk}
		Observe that the functor $I_{n+1}:\bbggO^{\gfrak_n}_{\bfrak_n}\rightsquigarrow\bbggO_{\bfrak_{n+1}}^{\gfrak_{n+1}}$ is right-adjoint to the forgetful functor $F_n:\bbggO^{\gfrak_{n+1}}_{\bfrak_{n+1}}\rightsquigarrow\bbggO_{\bfrak_n}^{\gfrak_n}$, and left-adjoint to $G_n:=\Hom_{\Ulie(\pfrak_{n+1})}\big(\Ulie(\gfrak_{n+1},\_\big):\bbggO^{\gfrak_{n+1}}_{\bfrak_{n+1}}\rightsquigarrow \bbggO^{\gfrak_n}_{\bfrak_n}$.  That is, 
		\begin{align}\Hom_{\bbggO^{\gfrak_{n+1}}_{\bfrak_{n+1}}}\left(M_{n+1},I_{n+1}M_n\right)\cong \Hom_{\bbggO^{\gfrak_{n}}_{\bfrak_{n}}}\left(F_nM_{n+1},M_n\right)\end{align} and  \begin{align}\Hom_{\bbggO^{\gfrak_{n+1}}_{\bfrak_{n+1}}}\left(I_{n+1}M_{n},M_{n+1}\right)\cong \Hom_{\bbggO^{\gfrak_{n}}_{\bfrak_{n}}}\left(M_{n},G_nM_{n+1}\right)\,,\end{align} for all $M_{n}\in\bbggO^{\gfrak_n}_{\bfrak_n}$ and $M_{n+1}\in\bbggO^{\gfrak_{n+1}}_{\bfrak_{n+1}}$. 
	\end{rmk}
	
	\begin{define}
		Fix $\lambda\in\hfrak^*$.  Define $R^\lambda_{n+1}:\bbggO^{\gfrak_{n+1}}_{\bfrak_{n+1}}\rightsquigarrow\bbggO^{\gfrak_{n+1}}_{\hfrak_{n+1}}[\lambda]$ to be the truncation functor, where for each $M_{n+1}\in\bbggO^{\gfrak_{n+1}}_{\bfrak_{n+1}}$, $R^\lambda_{n+1}M_{n+1}$ is the sum of all submodules $N_{n+1}$ of $M_{n+1}$ such that all composition factors of $N_{n+1}$ are not of the form $\Llie^{\gfrak_{n+1}}_{\bfrak_{n+1}}(\mu)$, with $\mu\in W_{\gfrak_n,\hfrak}\cdot\lambda$.
	\end{define}
	
	We have the following proposition.  The proof is trivial.
	
	\begin{prop}
		Fix $\lambda\in\hfrak^*$ and $n\in\Bbb Z_{>0}$.  
		\begin{enumerate}[(a)]
			\item The functor $R^\lambda_{n+1}:\bbggO^{\gfrak_{n+1}}_{\bfrak_{n+1}}\rightsquigarrow\bbggO^{\gfrak_{n+1}}_{\hfrak_{n+1}}[\lambda]$ is left-exact.
			\item The functor $R^\lambda_{n+1}I_{n+1}\operatorname{inj}^\lambda_{\gfrak_n,\bfrak_n}:\bbggO^{\gfrak_n}_{\bfrak_n}[\lambda]\to\bbggO^{\gfrak_{n+1}}_{\hfrak_{n+1}}[\lambda]$ is left-exact.
			\item The functor $R^\lambda_{n+1}I_{n+1}:\bbggO^{\gfrak_{n}}_{\bfrak_{n}}\rightsquigarrow\bbggO^{\gfrak_{n+1}}_{\hfrak_{n+1}}[\lambda]$ is left-exact.
		\end{enumerate}
	\end{prop}
	
	\begin{define}
		Let $Q^\lambda_{n+1}:\bbggO^{\gfrak_n}_{\bfrak_n}[\lambda]\rightsquigarrow\bbggO^{\gfrak_{n+1}}_{\bfrak_{n+1}}[\lambda]$ be given by
		\begin{align}
			Q^\lambda_{n+1}M_n:=I_{n+1}\operatorname{inj}^\lambda_{\gfrak_n,\bfrak_n}M_n/R^\lambda_{n+1}I_{n+1}\operatorname{inj}^\lambda_{\gfrak_n,\bfrak_n}M_n
		\end{align}
		for all $M_n\in\bbggO^{\gfrak_n}_{\bfrak_n}[\lambda]$.
	\end{define}
	
	\begin{thm}
		The functor $Q^\lambda_{n+1}:\bbggO^{\gfrak_n}_{\bfrak_n}[\lambda]\rightsquigarrow\bbggO^{\gfrak_{n+1}}_{\bfrak_{n+1}}[\lambda]$ is an exact functor.
	\end{thm}
	
	\begin{proof}
		Fix an exact sequence $0\to N_n\to M_n\to K_n\to 0$ of objects in $\bbggO^{\gfrak_n}_{\bfrak_n}[\lambda]$.  Because $W_{\gfrak_n,\hfrak}\cdot\lambda$ is finite, the length $k$ of the $\gfrak_n$-module $K_n$ is finite.  We shall prove by induction on $k$.
		
		For $k=0$, there is nothing to prove.  For $k=1$, we see that $K_n$ is a simple $\gfrak_n$-module.  Therefore, $K_n\cong\Llie^{\gfrak_n}_{\bfrak_n}(\mu)$ for some $\mu\in W_{\gfrak_n,\hfrak}\cdot\lambda$.  It can be easily seen that $Q_{n+1}^\lambda\Llie^{\gfrak_n}_{\bfrak_n}=\Llie^{\gfrak_{n+1}}_{\bfrak_{n+1}}(\mu)$.  Consider two exact sequences of $\gfrak_{n+1}$-modules:
		\begin{align}
			0\to I_{n+1}N_n\to I_{n+1}M_n\to I_{n+1}K_n\to 0
		\end{align}
		and
		\begin{align}
			0\to R^\lambda_{n+1}I_{n+1}\operatorname{inj}^\lambda_{\gfrak_n,\bfrak_n}N_n\to R^\lambda_{n+1}I_{n+1}\operatorname{inj}^\lambda_{\gfrak_n,\bfrak_n}M_n\to R^\lambda_{n+1}I_{n+1}\operatorname{inj}^\lambda_{\gfrak_n,\bfrak_n}K_n\,.
		\end{align}
		By definition of $Q^\lambda_{n+1}$, we obtain the following exact sequence
		\begin{align}
			0\to Q^\lambda_{n+1}N_n\to Q^\lambda_{n+1}M_n\to Q^\lambda_{n+1}K_n\,.
			\label{eq:exactly}
		\end{align}
		Because $Q^\lambda_{n+1}K_n$ is simple, either $Q^\lambda_{n+1}N_n\cong Q^\lambda_{n+1}M_n$ or the sequence
		\begin{align}
			0\to Q^\lambda_{n+1}N_n\to Q^\lambda_{n+1}M_n\to Q^\lambda_{n+1}K_n\to 0
			\label{eq:exactQ}
		\end{align}
		must be exact.
		
		Write $L_n:=\Llie^{\gfrak_n}_{\bfrak_n}(\mu)$ and $L_{n+1}:=\Llie^{\gfrak_{n+1}}_{\bfrak_{n+1}}(\mu)$.  Let $d$ denote the largest possible nonnegative integer such that there exists a quotient of $N_n$ isomorphic to $L_n^{\oplus d}$.  Then, $d+1$ is the largest possible nonnegative integer such that there exists a quotient of $M_n$ isomorphic to  $L_n^{\oplus(d+1)}$.  Then, there exists a $\gfrak_n$-submodule $Y_n$ of $N_n$ such that $N_n/Y_n\cong L_n^{\oplus d}$.  We claim that $Q^\lambda_{n+1}N_n/Q^\lambda_{n+1}Y_n\cong L_{n+1}^{\oplus d}$.
		
		As before, we have an exact sequence $0\to Q^\lambda_{n+1}Y_n\to Q^\lambda_{n+1}N_n\to L_{n+1}^{\oplus d}$.  This implies that $Q^\lambda_{n+1}N_n/Q^\lambda_{n+1}Y_n\cong L_{n+1}^{\oplus t}$ for some integer $t$ such that $0\le t\le d$.  If $t<d$, then $Q_{n+1}^\lambda Y_n$ has $L_{n+1}$ as a quotient.  Hence, $I_{n+1}Y_n$ has $L_{n+1}$ as a quotient.  Since $(I_{n+1}Y_n)^{\mu}$ is spanned by $1_{\Ulie(\gfrak_{n+1})}\underset{\Ulie(\pfrak_{n+1})}\otimes y_n$, where $y_n\in Y_n^\mu$, and by the action of $\pfrak_{n+1}$ on $Y_n$, we conclude that $L_{n}$ must also be a quotient of $Y_n$.  This contradict the definition of $d$.
		
		By a similar argument, if $X_n$ is a $\gfrak_n$-submodule of $M_n$ such that $M_n/X_n\cong L_n^{\oplus (d+1)}$, then $Q^\lambda_{n+1}M_n/Q^\lambda_{n+1}X_n\cong L_{n+1}^{\oplus (d+1)}$.  Thus, $Q^\lambda_{n+1}N_n\cong Q^\lambda_{n+1}M_n$ cannot hold.  Therefore, we must have an exact sequence~\eqref{eq:exactQ}.  
		
		Suppose now that $k>1$.  Consider two exact sequences of $\gfrak_n$-modules: $0\to N_n\to M_n\to K_n\to 0$ and $0\to Z_n\to K_n\to L_n\to 0$ for some simple object $L_n\in\bbggO^{\gfrak_n}_{\bfrak_n}[\lambda]$ and for some $\gfrak_n$-submodule $X_n$ of $K_n$.  Ergo, we can find exact sequences $0\to U_n\to M_n\to L_n\to 0$, $0\to N_n\to U_n\to Z_n\to 0$, where $U_n$ is a $\gfrak_n$-submodule of $M_n$.  By induction hypothesis,
		\begin{align}
			0\to Q^\lambda_{n+1}Z_n\to Q^\lambda_{n+1}K_n\to Q^\lambda_{n+1}L_n\to 0\,,
		\end{align}
		\begin{align}
			0\to Q^\lambda_{n+1}U_n\to Q^\lambda_{n+1}M_n\to Q^\lambda_{n+1}L_n\to 0\,,
		\end{align}
		and
		\begin{align}
			0\to Q^\lambda_{n+1}N_n\to Q^\lambda_{n+1}U_n\to Q^\lambda_{n+1}Z_n\to 0
		\end{align}
		are exact sequences.  Therefore,
		\begin{align}
			\ch\left(Q^\lambda_{n+1}M_n\right)&=\ch\left(Q^\lambda_{n+1}U_n\right)+\ch\left(Q^\lambda_{n+1}L_n\right)\nonumber\\&=\Bigg(\ch\left(Q^\lambda_{n+1}N_n\right)+\ch\left(Q^\lambda_{n+1}Z_n\right)\Bigg)+\ch\left(Q^\lambda_{n+1}L_n\right)\nonumber
			\\&=\ch\left(Q^\lambda_{n+1}N_n\right)+\Bigg(\ch\left(Q^\lambda_{n+1}Z_n\right)+\ch\left(Q^\lambda_{n+1}L_n\right)\Bigg)\nonumber\\&=\ch\left(Q^\lambda_{n+1}N_n\right)+\ch\left(Q^\lambda_{n+1}K_n\right)\,.
		\end{align}
		By \eqref{eq:exactly}, we conclude that $0\to Q^\lambda_{n+1}N_n\to  Q^\lambda_{n+1}M_n\to  Q^\lambda_{n+1}K_n\to 0$ must be an exact sequence.  The proof is now complete.
	\end{proof}
	
	\begin{cor}
		The functor $Q^{\lambda}_{n+1}$ is an equivalence between $\bbggO^{\gfrak_n}_{\bfrak_n}[\lambda]$ and the image $Q^{\lambda}_{n+1}\bbggO^{\gfrak_n}_{\bfrak_n}[\lambda]$.  More specifically, for any $\mu\in[\lambda]$ and $M_n\in\bbggO^{\gfrak_n}_{\bfrak_n}[\lambda]$, we have
		\begin{align}
			\left[M_n:\Llie^{\gfrak_n}_{\bfrak_n}(\mu)\right]=\left[Q_{n+1}^\lambda M_n:\Llie^{\gfrak_{n+1}}_{\bfrak_{n+1}}(\mu)\right]\,.
		\end{align}
		Furthermore, $Q_{n+1}^\lambda$ preserves the length of every object.
		\label{cor:preservation}
	\end{cor}
	
	\begin{prop}
		For every $M_n\in\bbggO^{\gfrak_n}_{\bfrak_n}[\lambda]$, there exists an injective $\gfrak_n$-module homomorphism $\iota_{M_n}:M_n\to Q^\lambda_{n+1}M_n$ such that, for all objects $M_n,N_n\in\bbggO^{\gfrak_n}_{\bfrak_n}[\lambda]$ along with a $\gfrak_n$-module homomorphism $f_n:M_n\to N_n$, the following diagram is commutative:
		\begin{equation}\begin{tikzcd}[column sep=large]
M_n\arrow{r}{f_n} \arrow[swap,dashed]{d}{\iota_{M_n}} & N_n \arrow[dashed]{d}{\iota_{N_n}} \\%
Q^\lambda_{n+1}M_n \arrow[swap]{r}{Q^\lambda_{n+1}f_n}& Q^\lambda_{n+1}N_n\,.
\end{tikzcd}
\end{equation}
	\label{prop:nattransQ}
	\end{prop}
	\begin{proof}
		For each $v\in M_n$, we define $\iota_{M_n}(v):=\left(1_{\Ulie(\gfrak_{n+1})}\underset{\Ulie(\pfrak_{n+1})}{\otimes} v\right)+R^\lambda_{n+1}I_{n+1}\operatorname{inj}^\lambda_{\gfrak_n,\bfrak_n}M_n\in Q^\lambda_{n+1}M_n$.  It is easy to see that $\iota_{M_n}$ satisfies the requirement.
	\end{proof}

	\begin{prop}
		Let $M_n\in\bbggO^{\gfrak_n}_{\bfrak_{n}}[\lambda]$ and $N_{n+1}\in\bbggO^{\gfrak_{n+1}}_{\bfrak_{n+1}}[\lambda]$.  Suppose that all composition factors of $N_{n+1}$ take the form $\Llie^{\gfrak_{n+1}}_{\bfrak_{n+1}}(\mu)$ with $\mu\in W_{\gfrak_n,\hfrak}\cdot\lambda$.  If $f:M_n\to N_{n+1}$ is a $\gfrak_n$-module homomorphism, then there exists a unique $\gfrak_{n+1}$-module homomorphism $\tilde{f}:Q^\lambda_{n+1}M_n\to N_{n+1}$ such that the following diagram is commutative:
		\begin{equation}\begin{tikzcd}
M_n\arrow{r}{f} \arrow[swap]{d}{\iota_{M_n}} & N_{n+1} \\%
Q^\lambda_{n+1}M_n \arrow[swap,dashed]{ur}{\tilde{f}}\,.
\end{tikzcd}
\end{equation}
\end{prop}
	
	\begin{proof}
		For each $u\in\Ulie(\gfrak_{n+1})$ and  $v\in M_n$, let the map $\tilde{f}$ send $\left(u\underset{\Ulie(\pfrak_{n+1})}{\otimes}v\right) +K_{n+1} \in Q^\lambda_{n+1}M_n$,  $K_{n+1}:=R^\lambda_{n+1}I_{n+1}\operatorname{inj}^\lambda_{\gfrak_n,\bfrak_n}M_n$, to $u\cdot f(v)\in N_{n+1}$.  Then, extend $\tilde{f}$ by linearity.  
		
		We claim that $\tilde{f}$ is a well defined homomorphism of $\gfrak_{n+1}$-modules.  Suppose that $u_1,u_2,\ldots,u_k$ are elements of $\Ulie(\gfrak_{n+1})$ and $v_1,v_2,\ldots,v_k$ are vectors in $M_n$ such that $\sum_{j=1}^k\,\left(u_j\underset{\Ulie(\pfrak_{n+1})}{\otimes} v_j\right) \in K_{n+1}$.  We want to prove that $\sum_{j=1}^k\,u_j\cdot f(v_j)=0$.   Write $z:=\sum_{j=1}^k\,u_j\cdot f(v_j)$. 
		
		The $\gfrak_{n+1}$-submodule $Z_{n+1}$ of $N_{n+1}$ generated by $z$ cannot have a composition factor of the form $\Llie^{\gfrak_{n+1}}_{\bfrak_{n+1}}(\xi)$ with $\xi\in W_{\gfrak_{n},\hfrak}\cdot\lambda$.  However, since all composition factors $N_{n+1}$ are of the form $\Llie^{\gfrak_{n+1}}_{\bfrak_{n+1}}(\mu)$ with $\mu\in W_{\gfrak_n,\hfrak}\cdot\lambda$, we conclude that $Z_{n+1}=0$.  Thus, $z=0$.
	\end{proof}
	
	\begin{prop}
		Let $M_n,N_n\in\bbggO^{\gfrak_n}_{\bfrak_n}[\lambda]$ and $f_{n+1}:Q^\lambda_{n+1}M_n\to Q^\lambda_{n+1}N_n$ be given.  Then, there exists a $\gfrak_n$-module homomorphism $\hat{Q}_n^\lambda f_{n+1}:M_n\to N_n$ such that the following diagram is commutative:
		\begin{equation}\begin{tikzcd}[column sep=large]
M_n\arrow[dashed]{r}{\hat{Q}_n^\lambda f_{n+1}} \arrow[swap]{d}{\iota_{M_n}} & N_n \arrow{d}{\iota_{N_n}} \\%
Q^\lambda_{n+1}M_n \arrow[swap]{r}{f_{n+1}}& Q^\lambda_{n+1}N_n\,.
\end{tikzcd}
\end{equation}
		For $M_{n+1}\in \im\left(Q^\lambda_{n+1}\right)$, suppose that $M_{n+1}=Q^\lambda_{n+1}M_n$ for some $M_n\in\bbggO^{\gfrak_n}_{\bfrak_n}[\lambda]$.  Define the $\gfrak_n$-module $\hat{Q}^\lambda_{n}M_{n+1}\in\bbggO^{\gfrak_n}_{\bfrak_n}[\lambda]$ to be $M_n$ itself.  Then, $\hat{Q}^\lambda_n:\im\left(Q^\lambda_{n+1}\right)\rightsquigarrow \bbggO^{\gfrak_n}_{\bfrak_n}[\lambda]$ is the inverse equivalence of $Q^\lambda_n: \bbggO^{\gfrak_n}_{\bfrak_n}[\lambda]\rightsquigarrow\im\left(Q^\lambda_{n+1}\right)$.
	\label{prop:qhat}
	\end{prop}
	
	\begin{proof}
		Let  $K_{n+1}:=R^\lambda_{n+1}I_{n+1}\operatorname{inj}^\lambda_{\gfrak_n,\bfrak_n}N_n$.  For each $v\in M_n$, suppose that $u_1,u_2,\ldots,u_k$ are elements of $\Ulie(\gfrak_{n+1})$ and $v_1,v_2,\ldots,v_k$ are vectors in $N_n$ such that  $f_{n+1}\big(\iota_{M_n}(v)\big)=\sum_{j=1}^k\,\left(u_j\underset{\Ulie(\pfrak_{n+1})}{\otimes} v_j\right) +K_{n+1}$.  Denote by $u'_j$ the projection of $u_j$ onto $\Ulie(\gfrak_n)$ (in the PBW basis of $\Ulie(\gfrak_n)$).  Set $f_n(v):=\sum_{j=1}^k\,u'_j\cdot v_j$.  Then, $\hat{Q}_n^\lambda f_{n+1}:=f_n$ satisfies the required condition.
	\end{proof}
	
	\begin{cor}
		Let $M_n,N_n\in\bbggO^{\gfrak_n}_{\bfrak_n}[\lambda]$.  Then,
		\begin{align}
			\Hom_{\bbggO^{\gfrak_n}_{\bfrak_n}[\lambda]}\left(M_n,N_n\right)\cong \Hom_{\bbggO^{\gfrak_{n+1}}_{\bfrak_{n+1}}[\lambda]}\left(Q^\lambda_{n+1}M_n,Q^\lambda_{n+1}N_n\right)\,.
		\end{align}
		Therefore, the image $Q^{\lambda}_{n+1}\bbggO^{\gfrak_n}_{\bfrak_n}[\lambda]$ is the full subcategory of $\bbggO^{\gfrak_{n+1}}_{\bfrak_{n+1}}$ whose objects have composition factors of the form $\Llie^{\gfrak_{n+1}}_{\bfrak_{n+1}}(\mu)$ with $\mu\in W_{\gfrak_n,\hfrak}\cdot\lambda$.
		\label{cor:homcong}
	\end{cor}
	
	For each $n\in\Bbb{Z}_{>0}$, fix a set $\mathcal{S}_n$ of representatives of $[\lambda]\in \Omega^{\gfrak_n}_{\bfrak_n}$.  We can further assume that $\mathcal{S}_n\supseteq\mathcal{S}_{n+1}$ for every positive integer $n$.  Then, $\mathcal{S}:=\bigcap\limits_{n\in\mathbb{Z}_{>0}}\,\mathcal{S}_n$ is a set of representatives of $[\lambda]\in\Omega^{\gfrak}_{\bfrak}$.
	
	\begin{prop}
		For each $\lambda\in \mathcal{S}$, the direct limit of $\Big(Q_{n+1}^\lambda:\bbggO_{\bfrak_n}^{\gfrak_n}[\lambda]\rightsquigarrow \bbggO_{\bfrak_{n+1}}^{\gfrak_{n+1}}[\lambda]\Big)_{n\in\Bbb Z_{>0}}$ is the full subcategory ${\mathcal{O}}^\gfrak_\bfrak[\lambda]$ of $\bbggO^\gfrak_\bfrak[\lambda]$ consisting of $\gfrak$-modules of finite length.
		\label{prop:limitofcalO}
	\end{prop}
	
	\begin{proof}
		For conveinence, write ${\mathcal{O}}[\lambda]$ for ${\mathcal{O}}^\gfrak_\bfrak[\lambda]$.  Fix $n\in\mathbb{Z}_{>0}$.  First, we define $\textup{q}^{\lambda}_n:\bbggO_{\bfrak_n}^{\gfrak_n}[\lambda]\rightsquigarrow{\mathcal{O}}[\lambda]$ as follows.  For a given $M_n\in\bbggO_{\bfrak_n}^{\gfrak_n}[\lambda]$ and $k\in\mathbb{Z}_{\geq 0}$, write $M_{n+k}$ for $Q^{\lambda}_{n+k}Q^{\lambda}_{n+k-1}\cdots Q_{n+1}^\lambda M_n \in \bbggO_{\bfrak_{n+k}}^{\gfrak_{n+k}}[\lambda]$.  Using Proposition~\ref{prop:nattransQ} above and noting that $\bbggO^{\gfrak_{n}}_{\bfrak_n}[\lambda]=\bggO^{\gfrak_n}_{\bfrak_n}[\lambda]$ (whose objects are finitely generated), we see that $\left(\iota_{M_{n+k}}:M_{n+k}\to M_{n+k+1}\right)_{k\in\mathbb{Z}_{\geq 0}}$ is a directed system whose direct limit $M$ is clearly in ${\mathcal{O}}[\lambda]$.  We set $\textup{q}_n^\lambda M_n$ to be the direct limit $M$.  It is easy to see that $\textup{q}_{n+1}^\lambda Q^\lambda_{n+1}=\textup{q}_n^\lambda$.
	
		Let $\textup{p}^\lambda_n:{\mathcal{O}}[\lambda]\rightsquigarrow \bbggO_{\bfrak_n}^{\gfrak_n}[\lambda]$ be the functor defined as follows: $
			\textup{p}^\lambda_nM:=\sum\limits_{\xi \in W_{\gfrak_n,\hfrak}\cdot\lambda}\,\Ulie(\gfrak_n)\cdot M^\xi
		$ for every $M\in{\mathcal{O}}[\lambda]$.  We can easily show that $\textup{p}^\lambda_n\textup{q}^\lambda_n=\textup{Id}_{\bbggO_{\bfrak_n}^{\gfrak_n}[\lambda]}$ and $\textup{p}^\lambda_{n+k}\textup{q}^\lambda_n=Q^\lambda_{n+k}Q^\lambda_{n+k-1}\cdots Q^\lambda_{n+1}$.  
		
		Suppose that there exists a category $\tilde{\mathcal{O}}[\lambda]$ along with functors $\tilde{\textup{q}}_n^\lambda:\bbggO_{\bfrak_n}^{\gfrak_n}[\lambda]\rightsquigarrow\tilde{\mathcal{O}}[\lambda]$ such that $\tilde{\textup{q}}_{n+1}^\lambda Q^\lambda_{n+1}=\tilde{\textup{q}}_{n}^\lambda$ for all $n=1,2,3,\ldots$.  Define $\textup{t}_n^\lambda:{\mathcal{O}}[\lambda]\rightsquigarrow\tilde{\mathcal{O}}[\lambda]$ via $\textup{t}^\lambda_nM:=\tilde{\textup{q}}^\lambda_n\textup{p}^\lambda_n M$ for all $M\in{\mathcal{O}}[\lambda]$.  Since $M$ is a $\gfrak$-module of finite length, $\textup{t}^\lambda_{n_0(M)} M=\textup{t}^\lambda_{n_0(M)+1}M=\textup{t}^\lambda_{n_0(M)+2}M=\ldots$ for some positive integer $n_0(M)$.  Let $\textup{t}^\lambda:{\mathcal{O}}[\lambda]\rightsquigarrow\tilde{\mathcal{O}}[\lambda]$ be given by $\textup{t}^\lambda M:=\textup{t}^\lambda_{n_0(M)}M$ for every $M\in{\mathcal{O}}[\lambda]$.  We can easily see that $\textup{t}^\lambda \textup{q}_n^\lambda=\tilde{\textup{q}}_n^\lambda$ for every positive integer $n$.  

		Note that the functor  $\textup{t}^\lambda:{\mathcal{O}}[\lambda]\rightsquigarrow\tilde{\mathcal{O}}[\lambda]$ above is unique with the property that $\textup{t}^\lambda \textup{q}_n^\lambda=\tilde{\textup{q}}_n^\lambda$ for every positive integer $n$.   Therefore, ${\mathcal{O}}[\lambda]$ is the required direct limit.
	\end{proof}
	
	\begin{cor}
		Let $\textup{q}^{\lambda}_n:\bbggO_{\bfrak_n}^{\gfrak_n}[\lambda]\rightsquigarrow{\mathcal{O}}^\gfrak_\bfrak[\lambda]$ be as given in the proof of the previous proposition.  Then, $\textup{q}^\lambda_n$ is an exact functor.
	\end{cor}

\begin{cor}
	Let $Q_{n+1}:\bbggO^{\gfrak_n}_{\bfrak_n}\rightsquigarrow \bbggO^{\gfrak_{n+1}}_{\bfrak_{n+1}}$ be the functor defined by
	\begin{align}
		Q_{n+1}M_n=\bigoplus_{\lambda\in \mathcal{S}}\,\operatorname{inj}^\lambda_{\gfrak_{n+1},\bfrak_{n+1}}Q^\lambda_{n+1}\pr^\lambda_{\gfrak_n,\bfrak_n} M_n
	\end{align}
	for all $M_n\in\bbggO^{\gfrak_n}_{\bfrak_n}$.  Then, the direct limit of $\Big(Q_{n+1}:\bbggO^{\gfrak_n}_{\bfrak_n}\rightsquigarrow \bbggO^{\gfrak_{n+1}}_{\bfrak_{n+1}}\Big)_{n\in\Bbb Z_{>0}}$ is the full subcategory ${\mathcal{O}}^{\gfrak}_{\bfrak}$ (or simply, ${\mathcal{O}}$) of $\bbggO^{\gfrak}_{\bfrak}$ along with a family of exact functors $\left(\textup{q}_n:\bbggO^{\gfrak_n}_{\bfrak_n}\rightsquigarrow {\mathcal{O}}^\gfrak_\bfrak\right)_{n\in\mathbb{Z}_{>0}}$, where ${\mathcal{O}}^\gfrak_\bfrak$ is given by ${\mathcal{O}}^{\gfrak}_{\bfrak}=\bigoplus\limits_{\lambda\in \mathcal{S}}\,{\mathcal{O}}^{\gfrak}_{\bfrak}[\lambda]$.
	\label{cor:qnfunctors}
\end{cor}


\subsection{Some Category Equivalences}

Take $\lambda\in\hfrak^*$.  We define the following notations:
\begin{itemize}
	\item $\Phi_{\gfrak,\hfrak}[\lambda]:=\Big\{\alpha\in \Phi\,\big|\,\lambda(h_\alpha)\in\mathbb{Z}\Big\}$ (also denoted by $\Phi[\lambda]$,
	\item  $\Phi^\pm_{\gfrak,\bfrak}[\lambda]:=\Phi_{\gfrak,\bfrak}[\lambda]\cap \Phi^{\pm}$ (also denoted by $\Phi^\pm[\lambda]$,
	\item $\Sigma_{\gfrak,\bfrak}[\lambda]$ or $\Sigma^+_{\gfrak,\bfrak}[\lambda]$ is the set of simple roots with respect to the set of positive roots $\Phi_{\gfrak,\bfrak}^+[\lambda]$ of the root system $\Phi_{\gfrak,\hfrak}[\lambda]$ (also denoted by $\Sigma[\lambda]$ or $\Sigma^+[\lambda]$,
	\item $\Sigma^-_{\gfrak,\bfrak}[\lambda]:=-\Sigma^+_{\gfrak,\bfrak}[\lambda]$ (also denoted by $\Sigma^-[\lambda]$),
	\item $\Lambda_{\gfrak,\hfrak}[\lambda]:=\huelle_\mathbb{Z}\Phi_{\gfrak,\hfrak}[\lambda]$ (also denoted by $\Lambda[\lambda]$,
	\item $\hfrak[\lambda]:=\huelle_\mathbb{K}\big\{h\in\hfrak\,|\,\lambda(h)\in\mathbb{Z}\rangle$,
	\item $\gfrak[\lambda]:=\hfrak[\lambda]\oplus\bigoplus_{\alpha\in\Phi[\lambda]}\,\gfrak^\alpha$,
	\item $\bfrak[\lambda]:=\bfrak[\lambda]\oplus\bigoplus_{\alpha\in\Phi^+[\lambda]}\,\gfrak^\alpha$ (also denoted by $\bfrak^+[\lambda]$),
	\item $\nfrak[\lambda]:=\bigoplus_{\alpha\in\Phi^+[\lambda]}\,\gfrak^\alpha$ (also denoted by $\nfrak^+[\lambda]$),
	\item $\nfrak^-[\lambda]:=\bigoplus_{\alpha\in\Phi^-[\lambda]}\,\gfrak^\alpha$ , 
	\item $\lambda^\natural:=\lambda\big|_{\hfrak[\lambda]}\in\big(\hfrak[\lambda]\big)^*$, and
	\item $\underline{W}_{\gfrak,\bfrak}[\lambda]$ (also denoted by $\underline{W}[\lambda]$) is the subgroup of $W_{\gfrak,\hfrak}$ consisting of elements $w$ such that $w\cdot \lambda=\lambda$.  
\end{itemize}

\begin{prop}
	Let $n$ be a positive integer and $\lambda\in\hfrak^*$.  Then, there exists a categorical equivalence $\mathcal{E}^\lambda_n:\bbggO^{\gfrak_n}_{\bfrak_n}[\lambda]\rightsquigarrow \bbggO^{\gfrak_n[\lambda]}_{\bfrak_n[\lambda]}[\lambda^\natural]$ which sends $\Llie^{\gfrak_n}_{\bfrak_n}(\mu)$ to $\Llie^{\gfrak_n[\lambda]}_{\bfrak_n[\lambda]}(\mu^\natural)$ for all $\mu\in[\lambda]\cap\left(W_{\gfrak_n,\hfrak}\cdot\lambda\right)$.
\end{prop}

\begin{proof}
	This proposition is a direct consequence of \cite[Theorem 11]{Soergel}.
\end{proof}

\begin{cor}
	 For each $\lambda\in\hfrak^*$, there exists a categorical equivalence $\mathcal{E}^\lambda:\mathcal{O}^{\gfrak}_{\bfrak}[\lambda]\rightsquigarrow \mathcal{O}^{\gfrak[\lambda]}_{\bfrak[\lambda]}[\lambda^\natural]$ which sends $\Llie^{\gfrak}_{\bfrak}(\mu)$ to $\Llie^{\gfrak[\lambda]}_{\bfrak[\lambda]}(\mu^\natural)$ for all $\mu\in[\lambda]$.
	 \label{cor:soergel}
\end{cor}

\begin{proof}
	This corollary follows from the previous proposition, Corollary~\ref{cor:preservation}, Proposition~\ref{prop:limitofcalO}, and Corollary~\ref{cor:homcong}.
\end{proof}

However, as a result of \cite[Theorem 11]{Soergel}, and Corollary~\ref{cor:soergel}, we have the following theorem.

\begin{thm}
	Let $\gfrak$ and $\gfrak'$ be root-reductive Lie algebras with Dynkin Borel subalgebras $\bfrak$ and $\bfrak'$, respectively.  Suppose that $S_{\gfrak,\bfrak}[\lambda]$ is the set of simple reflections with respect to elements of $\Sigma^+_{\gfrak[\lambda],\bfrak[\lambda]}$, and $S_{\gfrak',\bfrak'}[\lambda']$ is the set of simple reflections with respect to elements of $\Sigma^+_{\gfrak'[\lambda'],\bfrak'[\lambda']}$.  Suppose that there exists an isomorphism $\varphi:W_{\gfrak,\bfrak}[\lambda]\to W_{\gfrak',\bfrak'}[\lambda']$ of Coxeter systems $\big(W_{\gfrak,\bfrak}[\lambda],S_{\gfrak,\bfrak}[\lambda]\big)$ and $\big(W_{\gfrak',\bfrak'}[\lambda'],S_{\gfrak',\bfrak'}[\lambda']\big)$ such that 
	\begin{align}
		\varphi\big(\underline{W}_{\gfrak,\bfrak}[\lambda]\big)=\underline{W}_{\gfrak',\bfrak'}[\lambda']\,.
		\label{eq:conj}
	\end{align}
	Then, there exists an equivalence of categories $\mathcal{O}^\gfrak_\bfrak[\lambda]\cong\mathcal{O}^{\gfrak'}_{\bfrak'}[\lambda']$.
	\label{thm:conj}
\end{thm}

\begin{proof}
	By Corollary~\ref{cor:soergel}, we have $\mathcal{O}^\gfrak_\bfrak[\lambda]\cong\mathcal{O}^{\gfrak[\lambda]}_{\bfrak[\lambda]}[\lambda^\natural]$ and $\mathcal{O}^{\gfrak'}_{\bfrak'}[\lambda']\cong\mathcal{O}^{\gfrak'[\lambda']}_{\bfrak'[\lambda']}\big[(\lambda')^\natural\big]$.  Therefore, it suffices to assume that $\lambda$ and $\lambda'$ are both integral weights; that is, $\lambda=\lambda^\natural$, $\lambda'=(\lambda')^\natural$, $\gfrak[\lambda]=\gfrak$, $\bfrak[\lambda]=\bfrak$, $\gfrak'[\lambda']=\gfrak'$, and $\bfrak'[\lambda']=\bfrak'$.
	
	Let $n$ be a positive integer.  Define $S_{\gfrak_n,\bfrak_n}$ to be the set of simple reflections with respect to the elements of $\Sigma^+_{\gfrak_n,\bfrak_n}$.  The notation $S_{\gfrak'_n,\bfrak'_{n}}$ is defined similarly.  Set $\Sigma^+_{\gfrak_1,\bfrak_1}=\left\{\alpha_1,\alpha_2,\ldots,\alpha_{t_1}\right\}$, and for $n>1$, let $\Sigma^+_{\gfrak_n,\bfrak_n}=\Sigma^+_{\gfrak_{n-1},\bfrak_{n-1}}\cup\left\{\alpha_{t_{n-1}+1},\alpha_{t_{n-1}+2},\ldots,\alpha_{t_n}\right\}$.   Assume that $\varphi$ sends the simple reflection with respect to $\alpha_n$ to the simple reflection with respect to $\alpha'_{n}$ for every positive integer $n$. 

	Define $\gfrak''_n$ to be the subalgebra of $\gfrak'$ generated by $\hfrak'$ and the root spaces corresponding to the roots $\pm \alpha'_1$, $\pm\alpha'_2$, $\ldots$, $\pm\alpha'_{t_n}$.  Take $\bfrak''_n:=\gfrak''_n\cap\bfrak'$.  We note that the direct limits $\gfrak''$ and $\bfrak''$ of the directed systems $\left(\gfrak_n''\right)_{n\in\mathbb{Z}_{>0}}$ and $\left(\bfrak_n''\right)_{n\in\mathbb{Z}_{>0}}$ are precisely $\gfrak'$ and $\bfrak'$, respectively.  Hence, $\mathcal{O}^{\gfrak''}_{\bfrak''}[\lambda']=\mathcal{O}^{\gfrak'}_{\bfrak'}[\lambda']$.
	 
	The existence of $\varphi$ implies that, for each $n\in\mathbb{Z}_{>0}$, the Coxeter systems $\big(W_{\gfrak_n,\hfrak},S_{\gfrak_n,\bfrak_n}\big)$ and  $\big(W_{\gfrak''_{n},\hfrak'},S_{\gfrak''_{n},\bfrak''_{n}}\big)$ are isomorphic.  Therefore, by  \cite[Theorem 11]{Soergel},  there exists an equivalence of categories $\varepsilon_n:\bggO^{\gfrak_n}_{\bfrak_n}[\lambda]\rightsquigarrow\bggO^{\gfrak'_{k_n}}_{\bfrak'_{k_n}}[\lambda']$.  Applying direct limit, we obtain an equivalence $\varepsilon:\mathcal{O}^{\gfrak}_{\bfrak}[\lambda] \rightsquigarrow \mathcal{O}^{\gfrak'}_{\bfrak'}[\lambda']$.
\end{proof}

\begin{openq}
	Is the converse of Theorem~\ref{thm:conj} true?  In other words, if $\mathcal{O}^\gfrak_\bfrak[\lambda]\cong\mathcal{O}^{\gfrak'}_{\bfrak'}[\lambda']$, then does there exists an isomorphism  $\varphi:W_{\gfrak,\bfrak}[\lambda]\to W_{\gfrak',\bfrak'}[\lambda']$ of Coxeter systems $\big(W_{\gfrak,\bfrak}[\lambda],S_{\gfrak,\bfrak}[\lambda]\big)$ and $\big(W_{\gfrak',\bfrak'}[\lambda'],S_{\gfrak',\bfrak'}[\lambda']\big)$ such that \eqref{eq:conj} is true?
\end{openq}

\subsection{Construction of Translation Functors}

\begin{define}
	For $\lambda\in\hfrak^*$, we say that $\lambda$ is \emph{restricted} if $\lambda(h_\alpha)=0$ for all but finitely many $\alpha\in\Sigma_{\gfrak,\bfrak}^+$.  For $\lambda,\mu\in\hfrak^*$, we say that $\lambda$ and $\mu$ are \emph{compatible} if $\lambda-\mu\in\Lambda$ and $\lambda-\mu$ is a restricted weight.  The notation $\lambda\parallel \mu$ means that $\lambda$ and $\mu$ are compatible.
\end{define}

Denote by $\hfrak_\mathbb{Q}$ the $\mathbb{Q}$-span of the coroots $h_\alpha\in \hfrak$ with $\alpha\in\Phi=\Phi_{\gfrak,\hfrak}$.  Take $\hfrak_\mathbb{R}$ to be $\mathbb{R}\underset{\mathbb{Q}}{\otimes}\hfrak_\mathbb{Q}$.  Let $E:=E^\gfrak_\hfrak$ denote the real vector space $\Hom_\mathbb{R}\left(\hfrak_\mathbb{R},\mathbb{R}\right)$.  A root $\alpha\in\Phi$ can be identified with the unique element (which we also denote by $\alpha$) of $E$ which sends $1\underset{\mathbb{Q}}{\otimes}\beta\in \hfrak_\mathbb{R}$ to $\alpha(h_\beta)\in\mathbb{Z}$ for all $\beta\in\Phi$.    Similarly, if $\lambda\in\hfrak^*$ satisfies $\lambda(h_\beta)\in\mathbb{Q}$ for all $\beta\in\Phi$, then we identify it with the unique element of $E$ that sends $h_\beta\mapsto \lambda(h_\beta)$ for all $\beta\in\Phi$.

We decompose $E$ into facets, where a \emph{facet} $F$ of $E$ is a nonempty subset of $E$ determined by the partition of $\Phi$ into disjoint subsets $\Phi^+(F)$, $\Phi^0(F)$, and $\Phi^-(F)$, where $\lambda\in F$ if and only if all three conditions below are satisfied:
\begin{itemize}
	\item $(\lambda+\rho)(h_\alpha)>0$ when $\alpha\in \Phi^+(F)$,
	\item $(\lambda+\rho)(h_\alpha)=0$ when $\alpha\in \Phi^0(F)$, and
	\item $(\lambda+\rho)(h_\alpha)<0$ when $\alpha\in \Phi^-(F)$.
\end{itemize}
The closure $\bar{F}$ of a facet $F$ is defined to be the set of all $\lambda\in\hfrak^*$ such that
\begin{itemize}
	\item $(\lambda+\rho)(h_\alpha)\geq 0$ when $\alpha\in \Phi^+(F)$,
	\item $(\lambda+\rho)(h_\alpha)=0$ when $\alpha\in \Phi^0(F)$, and
	\item $(\lambda+\rho)(h_\alpha)\leq 0$ when $\alpha\in \Phi^-(F)$.
\end{itemize}

Recall that $\mathcal{S}$ is a (fixed) set of representatives of $[\lambda]\in\Omega^\gfrak_\bfrak$ (on which the definitions of $Q_{n+1}$ and $\text{q}_n$ depend).  Now, for $\lambda,\mu\in\mathcal{S}$ such that $\lambda\parallel \mu$, there exists a unique dominant-integral weight $\nu\in W_{\gfrak,\hfrak}(\lambda-\mu)$.  Define for all sufficiently large $n$ (i.e., for all positive integers $n$ such that $\nu\in W_{\gfrak_n,\hfrak_n}(\lambda-\mu)$) the translation functor $\left(\theta^{\gfrak_n}_{\bfrak_n}\right)^\mu_\lambda:\bbggO^{\gfrak_n}_{\bfrak_n}[\lambda]\rightsquigarrow\bbggO^{\gfrak_n}_{\bfrak_n}[\mu]$.  That is,
\begin{align}
	\left(\theta^{\gfrak_n}_{\bfrak_n}\right)^\mu_\lambda M_n\defeq \pr^\mu_{\gfrak_n,\bfrak_n}\big(\Llie^{\gfrak_n}_{\bfrak_n}(\nu)\otimes M_n\big)
\end{align}
for every $M_n\in\bbggO^{\gfrak_n}_{\bfrak_n}[\lambda]$.  Recall that $\left(\theta^{\gfrak_n}_{\bfrak_n}\right)^\mu_\lambda$ is an exact functor and it commutes with duality (see, for example, \cite[Proposition 7.1]{bggo}).

\begin{thm}
	Let $\lambda,\mu\in\hfrak^*$ be such that $\lambda\parallel \mu$.   If $\lambda^\natural$ and $\mu^\natural$ are in the same facet for the action of the integral Weyl group $W_{\gfrak,\bfrak}[\lambda]=W_{\gfrak,\bfrak}[\mu]$ on $E^{\gfrak[\lambda]}_{\hfrak[\lambda]}=E^{\gfrak[\mu]}_{\hfrak[\mu]}$, then there exists an equivalence of categories $\left(\Theta^{\gfrak}_{\bfrak}\right)^{\mu}_{\lambda}:\mathcal{O}^\gfrak_\bfrak[\lambda]\rightsquigarrow\mathcal{O}^\gfrak_\bfrak[\mu]$.
\end{thm}

\begin{proof}
For convenience, write $\bbggO_n$ for $\bbggO^{\gfrak_n}_{\bfrak_n}$.  For $\zeta,\xi\in\hfrak^*$, we also denote by $\left(\theta_n\right)^\zeta_\xi$ the functor $\left(\theta^{\gfrak_n}_{\bfrak_n}\right)^\zeta_\xi$.  We also write $W$ and $W_n$ for $W_{\gfrak,\hfrak}$ and $W_{\gfrak_n,\hfrak_n}$, respectively.  

Let $n_0$ be the smallest integer such that $\nu\in W_{n_0}(\lambda-\mu)$.  For each integer $n\geq n_0$, let $\xi_n\in W_{n}[\lambda]\cdot\lambda$ be the antidominant weight that is linked to $\lambda$,  and let $\zeta_n\in W_{n}[\mu]\cdot\mu$ be the antidominant weight that is linked to $\mu$.

Since $\lambda^\natural$ and $\mu^\natural$ are in the same facet  for the action of $W[\lambda]=W[\mu]$ on $E^{\gfrak[\lambda]}_{\hfrak[\lambda]}=E^{\gfrak[\mu]}_{\hfrak[\mu]}$,  there exists $w_n\in W_n[\lambda]=W_n[\mu]$ such that $w_n\cdot\lambda = \xi_n$ and $w_n\cdot\mu=\zeta_n$.  From \cite[Theorem 7.8]{bggo}, we conclude that 
\begin{align}\Theta_n:=\left(\theta_n\right)^{\zeta_n}_{\xi_n}:\bbggO_n[\lambda]\rightsquigarrow\bbggO_n[\mu]\end{align} is an equivalence of categories whose equivalence is \begin{align}\Theta'_n:=(\theta_n)^{\xi_n}_{\zeta_n}:\bbggO_n[\mu]\rightsquigarrow\bbggO[\lambda]\,.\end{align}

 Recall the functors $\hat{Q}^\lambda_n:\im\left(Q^\lambda_{n+1}\right)\rightsquigarrow \bggO^\lambda_n$ and $\hat{Q}^\lambda_n:\im\left(Q^\lambda_{n+1}\right)\rightsquigarrow \bggO^\lambda_n$ and $\hat{Q}^\mu_n:\im\left(Q^\mu_{n+1}\right)\rightsquigarrow \bggO^\mu_n$ from Proposition~\ref{prop:qhat}.  Furthermore, due to \cite[Proposition 7.8]{bggo} any object in $\im\left(\Theta_{n+1}Q_{n+1}^\lambda\right)$ must have composition factors of the form $\Llie^{\gfrak_{n+1}}_{\bfrak_{n+1}}\left(w\cdot\zeta_n\right)$, where $w \in W_n[\mu]\cdot \mu$.  Therefore, $\im\left(\Theta_{n+1}Q_{n+1}^\lambda\right)$ is a subcategory of $\im\left(Q_{n+1}^\mu\right)$.   Let
\begin{align}
\pazocal{E}_{n}:=\Theta_n'\hat{Q}_n^\mu\Theta_{n+1}Q^\lambda_{n+1}:\bbggO_n[\lambda]\rightsquigarrow\bbggO_n[\lambda]\,.
\end{align}
Clearly, $\pazocal{E}_n$ is an auto-equivalence of $\bbggO_n[\lambda]$ (since $Q^\lambda_{n+1}$ is an equivalence on to its image, $\Theta_{n+1}$ and $\Theta_n$ are both equivalences, and $\hat{Q}^\mu_n$ is an equivalence).  Consequently, 
\begin{align}\Theta_{n+1} Q^\lambda_{n+1}= Q^\mu_{n+1}  \Theta_n \pazocal{E}_n\cong Q^\mu_{n+1}\Theta_n\,.\end{align}
Write $\mathcal{O}$ for $\mathcal{O}^\gfrak_\bfrak$.  We can now let $\Theta:\mathcal{O}[\lambda]\rightsquigarrow\mathcal{O}[\mu]$ be the direct limit of the directed system of functors $\left(\Theta_n:\bbggO_n[\lambda]\rightsquigarrow\bbggO_n[\mu]\right)_{n\in\mathbb{Z}_{>0}}$.  Similarly,  $\Theta':\mathcal{O}[\mu]\rightsquigarrow\mathcal{O}[\lambda]$ is the direct limit of the directed system of functors $\left(\Theta'_n:\bbggO_n[\mu]\rightsquigarrow\bbggO_n[\lambda]\right)_{n\in\mathbb{Z}_{>0}}$.  As each $\Theta_n$ is an equivalence of categories with inverse $\Theta'_n$, we deduce that $\Theta$ is also an equivalence of categories with inverse $\Theta'$.  We set $\left(\Theta^\gfrak_\bfrak\right)^\mu_\lambda$ to be the functor $\Theta$.
\end{proof}

From the previous theorem, we have defined a "translation functor" $\left(\Theta^\gfrak_\hfrak\right)^{\mu}_{\lambda}$ when $\mu$ and $\lambda$ are compatible and lie in the same facet for the action of $W_{\gfrak,\bfrak}[\lambda]=W_{\gfrak,\bfrak}[\mu]$.  For arbitrary $\lambda,\mu\in\hfrak^*$ such that $\lambda\parallel \mu$, it is not clear whether the same construction yields a functor $\left(\Theta^\gfrak_\hfrak\right)^\mu_\lambda:\mathcal{O}^\gfrak_\bfrak[\lambda]\rightsquigarrow\mathcal{O}^\gfrak_\bfrak[\mu]$.  


\section{Tilting Modules}

\subsection{Extensions of Modules with Generalized Standard and Costandard Filtrations}

In this subsection, we shall write $\Hom$ and $\Ext$ for $\Hom_\bbggO$ and $\Ext_\bbggO$.   We shall first prove that any extension of a module with generalized costandard filtration by a module with generalized standard filtration is trivial.

\begin{thm}
	Let $M\in\boldsymbol{\Delta}(\bbggO)$ and $N\in \boldsymbol{\nabla}(\bbggO)$.  Then, $\Ext^1(M,N)=0$.
	\label{thm:extstandardcostandard}
\end{thm}

\begin{proof}
	If $M$ has a direct sum decomposition $M=\bigoplus_{\alpha\in A}\,M_\alpha$, then $\Ext^1(M,N)\cong\prod_{\alpha \in A}\,\Ext^1(M_\alpha,N)$.  Thus, we may assume without loss of generality that $M$ is indecomposable.  
	
	We first prove the theorem when $N=\nabla(\mu)$ for some $\mu \in \hfrak^*$.  Let $\Pi_{\succ \mu}(M)$ denote the set of weights of $M$ that are greater than $\mu$.  Note that $M$ is in the block $\bbggO[\lambda]$ for some $\lambda\in\hfrak^*$.  If $\Pi_{\succ \mu}(M)$ has infinitely many maximal elements, then we enumerate the maximal elements of $\Pi_{\succ \mu}(M)$ by $\lambda_1,\lambda_2,\lambda_3,\ldots$.  Note that $\lambda_i\in W[\lambda]\cdot\lambda$ for all $i=1,2,3,\ldots$.  This implies that $\lambda-\mu$ is not a finite integer combination of the simple roots.  Therefore, $\mu$ and $\lambda$ are not in the same Weyl orbit.  Thus, $\Ext^1(M,N)=\Ext^1\big(M,\nabla(\mu)\big)=0$.  From now on, we assume that $\Pi_{\succ \mu}(M)$ has finitely many maximal elements.
	
	We perform induction on the sum $m:=\sum_{\xi\succ \mu}\,\dim M^\xi$, which is finite due to the assumption in the previous paragraph.  If $m=0$, then by \cite[Proposition 3.8(a)]{DVN1} and \cite[Proposition 3.9(a)]{DVN1}, we get $\Ext^1\big(M,\nabla(\mu)\big)\cong \Ext^1\big(\Delta(\mu),M^\vee\big)=0$. 
	
	Let now $m$ be a positive integer.  Fix a maximal weight $\xi\in \Pi_{\succ \mu}$.  By Proposition~\ref{prop:standardfiltration}(a), there exists a short exact sequence $0\to M' \to M \to M'' \to 0$, where $M'$ is isomorphic to $\Delta(\xi)$, and $M''$ has a generalized standard filtration.  From the long exact sequence of $\Ext^\bullet$, we get the following exact sequence
	\begin{align}
		\Ext^1\big(M'',\nabla(\mu)\big)\to \Ext^1\big(M,\nabla(\mu)\big)\to \Ext^1\big(M',\nabla(\mu)\big)\,.
	\end{align}
	However, due to \cite[Proposition 3.9(c)]{DVN1}, we get $ \Ext^1\big(M',\nabla(\mu)\big)\cong \Ext^1\big(\Delta(\xi),\nabla(\mu)\big)=0$.  By induction hypothesis, $\Ext^1\big(M'',\nabla(\mu)\big)=0$.  Therefore, $\Ext^1\big(M,\nabla(\mu)\big)=0$ as well.
	
	For each ordinal number $\gamma$, we shall define a submodule $X_\gamma \in \boldsymbol{\Delta}(\bbggO)$ of $N^\vee$ such that $N^\vee /X_\gamma$ is also in $\boldsymbol{\Delta}(\bbggO)$.  First, we set $X_0:=0$.  If $\gamma$ is an ordinal with predecessor $\gamma'$, then we have by Proposition~\ref{prop:standardfiltration}(a) that $N^\vee/X_{\gamma'}$ has a submodule $Y_\gamma$ such that $Y_\gamma\cong \Delta(\mu_\gamma)$ for some $\mu_\gamma\in\hfrak^*$ that is a maximal weight of $N^\vee/X_{\gamma'}$.  Take $X_\gamma$ to be the preimage of $Y_\gamma$ under the canonical projection $N^\vee \twoheadrightarrow \left(N^\vee/X_{\gamma'}\right)$.   If $\gamma$ is a limit ordinal, then we set $X_\gamma$ to be $\bigcup_{\gamma'<\gamma}\,X_{\gamma'}$.
	
	From the above construction, there exists the smallest ordinal $\kappa$ such that $N^\vee =X_\kappa$ , and $N^\vee$ is in fact the direct limit of $(X_\gamma)_{\gamma\leq \kappa}$.  Thus, $N$ is the inverse limit of $(N_\gamma)_{\gamma\leq \kappa}$, where $N_{\gamma}:=\left(X_\gamma\right)^\vee$.  We claim that $\Ext^1\big(M,N_\gamma\big)=0$ for all ordinals $\gamma\leq \kappa$.
	
	If $\gamma$ has a predecessor $\gamma'$, then there exists a short exact sequence $0\to \nabla(\mu_\gamma)\to N_\gamma\to N_{\gamma'}\to 0$.  From the long exact sequence of $\Ext^\bullet$, we have the following exact sequence
	\begin{align}
		 \Ext^1\big(M,\nabla(\mu_\gamma)\big)\to \Ext^1\big(M,N_\gamma\big)\to \Ext^1\big(M,N_{\gamma'}\big)\,.
		 \label{eq:longextexactgamma}
	\end{align}
	By the induction hypothesis, $\Ext^1\big(M,N_{\gamma'}\big)=0$.  Since we have proven that $ \Ext^1\big(M,\nabla(\mu_\gamma)\big)=0$, we can then conclude that $\Ext^1\big(M,N_\gamma\big)=0$.  If $\gamma$ is a limit ordinal, then $X_\gamma=\lim_{\substack{\longrightarrow\\\gamma'<\gamma}}\,X_{\gamma'}$; ergo,
	\begin{align}
		\Ext^1\big(M,N_\gamma\big)\cong \Ext^1\big((N_\gamma)^\vee,M^\vee\big) = \Ext^1\big(X_\gamma,M^\vee\big)\,.
		\label{eq:extisomorphism}
	\end{align}	
	Now, let $E$ be a $\gfrak$-module such that there exists an exact sequence $0\to M^\vee \overset{i}{\longrightarrow} E \overset{p}{\longrightarrow} X_\gamma\to 0$.   For each $\gamma'<\gamma$, we know that $\Ext^1\big(X_{\gamma'},M^\vee\big)\cong \Ext^1\big(M,N_{\gamma'}\big)=0$ by induction hypothesis, whence $p^{-1}\left(X_{\gamma'}\right)=i(M^\vee)\oplus \tilde{X}_{\gamma'}$, where $M_{\gamma'}$ and $\tilde{X}_{\gamma'}$ are submodules of $p^{-1}\left(X_{\gamma'}\right)$ such that $M_{\gamma'}\cong M^\vee$ and $\tilde{X}_{\gamma'}\cong X_{\gamma'}$.   We shall now define $v_{\gamma'}\in \mathscr{X}_{\gamma'}$ whenever $\gamma'$ has a predecessor $\gamma''$.  
	
	If $\gamma'=1$, then $X_{\gamma'}\cong\Delta(\mu_{\gamma'})$ for some $\mu_{\gamma'}\in \hfrak^*$.  For each ordinal $\delta$ such that $\gamma'\leq \delta<\gamma$, define $\mathcal{X}^\delta_{\gamma'}$ to be the span of all possible $v^\delta_{\gamma'}\in p^{-1}(X_{\gamma'})$ such that $v_{\gamma'}^\delta$ is a weight vector of weight $\mu_{\gamma'}$ that lies in some submodule $\tilde{X}_{\delta}\cong X_\delta$ of $p^{-1}(X_\delta)$ such that $p^{-1}(X_\delta)=i(M^\vee)\oplus \tilde{X}_{\delta}$ and that $v^\delta_{\gamma'}$ generates $\tilde{X}_{\gamma'}\cong\Delta(\mu_{\gamma'})$.  Note that $\mathcal{X}^\delta_{\gamma'}$ is a finite-dimensional vector space with positive dimension, and $\mathcal{X}^{\delta'}_{\gamma'}\supseteq \mathcal{X}^{\delta}_{\gamma'}$  if $\gamma'<\delta'<\delta<\gamma$.  Therefore, $\mathcal{X}_{\gamma'}:=\bigcap_{\delta\in[\gamma',\gamma)}\,\mathcal{X}^{\delta}_{\gamma'}\neq 0$.  We can choose $v_{\gamma'}\in \mathcal{X}_{\gamma'}\setminus\{0\}$ arbitrarily.
	
	Let now $\gamma'$ be an ordinal such that $1<\gamma'<\gamma$ and $\gamma'$ has a predecessor $\gamma''$.  Suppose $v_\tau$ are all known for each $\tau<\gamma'$ such that $\tau$ has a predecessor.  Set $Z_{\gamma''}$ to be the $\gfrak$-module generated by all such $v_\tau$.  The choices of our vectors $v_\tau$ are to ensure that $Z_{\gamma''}\cong X_{\gamma''}$ is such that $p^{-1}(X_{\gamma''})=i(M^\vee)\oplus Z_{\gamma''}$.  Assume that $X_{\gamma'}/X_{\gamma''}\cong\Delta(\mu_{\gamma'})$ for some $\mu_{\gamma'}\in \hfrak^*$.   For each ordinal $\delta$ such that $\gamma'\leq \delta<\gamma$, define $\mathcal{X}^\delta_{\gamma'}$ to be the span of all possible $v^\delta_{\gamma'}\in p^{-1}(X_{\gamma'})$ such that $v_{\gamma'}^\delta$ is a weight vector of weight $\mu_{\gamma'}$ that lies in some submodule $\tilde{X}_{\delta}\cong X_\delta$ of $p^{-1}(X_\delta)$ such that $p^{-1}(X_\delta)=i(M^\vee)\oplus \tilde{X}_{\delta}$ and that $v^\delta_{\gamma'}+Z_{\gamma'}$ generates $\tilde{X}_{\gamma'}/Z_{\gamma''}\cong\Delta(\mu_{\gamma'})$.  We employ the same strategy as the previous paragraph by choosing $v_{\gamma'}\in\mathcal{X}_{\gamma'}\setminus Z_{\gamma''}$, where $\mathcal{X}_{\gamma'}:=\bigcap_{\delta\in[\gamma',\gamma)}\,\mathcal{X}^{\delta}_{\gamma'}$.
	
	Now, we let $\tilde{X}_\gamma$ be the submodule of $E$ generated by $v_{\gamma'}$ for all ordinals $\gamma'<\gamma$ with predecessors.  It follows immediately that $\tilde{X}_\gamma\cong X_\gamma$, $i(M^\vee)\cap \tilde{X}_\gamma=0$, and $i(M^\vee)+\tilde{X}_{\gamma}=E$.  Thus, $E=i(M^\vee)\oplus\tilde{X}_\gamma$.  Then, the projection $\varpi:E\to i(M^\vee)$ gives a retraction map $E\to M^\vee$.  Therefore, the short exact sequence $0\to M^\vee \to E \to X_{\gamma}\to 0$ splits.  That is, $\Ext^1(X_\gamma,M^\vee)=0$.  Then, \eqref{eq:extisomorphism} implies that $\Ext^1(M,N_\gamma)=0$ as well.  By transfinite induction, $\Ext^1(M,N)=\Ext^1(M,N_\kappa)=0$.
\end{proof}

\begin{conj}
	Let $M\in\boldsymbol{\Delta}(\bbggO)$ and $N\in \boldsymbol{\nabla}(\bbggO)$.  Then, $\Ext^k(M,N)=0$ for every integer $k>1$.
	\label{conj:extstandardcostandard}
\end{conj}


\begin{prop}
	Let $M\in\bbggO$ be a tilting module.  
	\begin{enumerate}[(a)]
		\item The dual $M^\vee$ is also a tilting module.
		\item If $N\in\bbggO$ is a tilting module, then $M\oplus N$ is also a tilting module.
		\item  Any direct summand of $M$ is a tilting module.
	\end{enumerate}
		\label{prop:tiltingprops}
\end{prop}

\begin{proof}
Parts (a) and (b) are trivial.  Part (c) follows from Proposition~\ref{prop:standardfiltration}(b) and Corollary~\ref{cor:costandardfiltration}.  
\end{proof}

\begin{prop}
	Let $M$ and $N$ be tilting modules in $\bbggO$.  Then, $\Ext^1_{\bbggO}(M,N)=0$.  If Conjecture~\ref{conj:extstandardcostandard} is true, then we also have that $\Ext^k_{\bbggO}(M,N)=0$ for all integers $k>1$.
	\label{prop:tiltingext}
\end{prop}

\subsection{Construction of the Tilting Modules $D(\lambda)$}

\begin{prop}
	Let $\lambda,\mu\in\hfrak^*$.  Then, $\Delta(\lambda)\otimes \nabla(\mu)$ is a tilting module in $\bbggO$.
\end{prop}

\begin{proof}
	Let $M:=\Delta(\lambda)\otimes \nabla(\mu)$.   First, observe that,  for all $\nu\preceq \lambda+\mu$, we have
	\begin{align}
		\dim (M^\nu)=\sum_{\xi\preceq 0}\,\dim \Big(\big(\Delta(\lambda)\big)^{\nu-\mu+\xi}\Big)\cdot \dim \Big(\big(\Delta(\mu)\big)^{\mu-\xi}\Big)
	\end{align} 
	Because there are only finitely many weights of $\Delta(\lambda)$ that is greater than or equal to $\nu-\mu$, we see that $\dim (M^\nu)<\infty$.  Therefore, $M\in\bbggO$.
	
	Since $\big(\Delta(\lambda)\otimes \nabla(\mu)\big)^\vee \cong \big(\Delta(\lambda)\big)^\vee\otimes \big(\nabla(\mu)\big)^\vee \cong \nabla(\lambda)\otimes \Delta(\mu)\cong \Delta(\mu)\otimes \nabla(\lambda)$, it suffices to show that $M:=\Delta(\lambda)\otimes \nabla(\mu)$ has a generalized standard filtration.  
	
	Let $u$ be a maximal vector of $\Delta(\lambda)$.  Pick a basis $v_1,v_2,v_3,\ldots$ of $\nabla(\mu)$ consisting of weight vectors.  Then, define $w_i:=u\otimes v_i$ for $i=1,2,3,\ldots$.  We first prove that $w_1,w_2,w_3,\ldots$ generate $M$ as a $\Ulie(\nfrak^-)$-module.   Let $M'$ be the $\Ulie(\nfrak^-)$-submodule of $M$ generated by $w_1,w_2,w_3,\ldots$.
	
	Fix a Chevalley basis of $\gfrak$ consisting of $x_{\pm \alpha}$ for positive roots $\alpha$, and $h_\alpha$ for simple positive roots $\alpha$.  We then fix a PBW basis $B$ of $\Ulie(\nfrak^-)$.  For each $t\in B$, the \emph{degree} of $t$, denoted by $\deg(t)$, is defined to be $k$ if there exists positive roots $\alpha_1,\alpha_2,\ldots,\alpha_k$ such that $t=x_{-\alpha_1}x_{-\alpha_2}\cdots x_{-\alpha_k}$.
	
	We shall prove that $(t\cdot u)\otimes v_i \in M'$.  If $\deg(t)=0$, then there is nothing to prove.  Suppose now that $\deg(t)>0$.  Then, we can write
	\begin{align}t=x_{-\alpha_1}x_{-\alpha_2}\cdots x_{-\alpha_k}\end{align}
	for some integer $k>0$.  By induction hypothesis, we know that $m:=(x_{-\alpha_2}\cdots x_{-\alpha_k}\cdot u)\otimes v_i$ lies in $M'$.  Using 
	\begin{align}
		x_{-\alpha_1}\cdot m=(t\cdot u)\otimes v_i+(x_{-\alpha_2}\cdots x_{-\alpha_k}\cdot u)\otimes (x_{-\alpha_1}\cdot v_i)\,,
	\end{align}
	we conclude that $(t\cdot u)\otimes v_i$ is in $M'$, as both $x_{-\alpha_1}\cdot m$ and $(x_{-\alpha_2}\cdots x_{-\alpha_k}\cdot u)\otimes (x_{-\alpha_1}\cdot v_i)$ are in $M'$.
	
	From the paragraph above, $M=M'$.  We now need to show that $M$ is a free module over $\Ulie(\nfrak^-)$ generated by $w_1,w_2,w_3,\ldots$.  Let now $M_k$ denote the $\Ulie(\nfrak^-)$-submodule of $M$ generated by $w_1,w_2,\ldots,w_k$, and $N_k$ the $\Ulie(\nfrak^-)$ submodule of $N$ generated by $w_k$ alone.  Then, we can easily see that $M_k\cap N_{k+1}=0$ for each $k=1,2,3,\ldots$.  Thus, $M_k=N_1\oplus N_2\oplus \ldots \oplus N_k$, making
	\begin{align}
		M=N_1\oplus N_2\oplus N_3\oplus\ldots
	\end{align}
	as a $\Ulie(\nfrak^-)$-module.  Consequently, $M$ has a generalized standard filtration.
\end{proof}

\begin{thm}
	Let $\lambda\in\hfrak^*$.  There exists a unique, up to isomorphism, an indecomposable tilting module $D^\gfrak_\bfrak(\lambda)\in \bbggO$, also denoted by $D(\lambda)$, such that $\dim \Big(\big(D(\lambda)\big)^\lambda\Big)=1$ and all weights $\mu$ of $D(\lambda)$ satisfies $\mu\preceq\lambda$.
\end{thm}

\begin{proof}
	Consider the $\gfrak$-module $M:=\Delta(\lambda)\otimes \nabla(0)$.	Define $D(\lambda)$ to be the indecomposable summand of $M$ that contains $M^\lambda$.  By Proposition \ref{prop:tiltingprops}(c), we know that $D(\lambda)$ is a tilting module.
	
	Suppose $T$ is another indecomposable tilting module such that $\dim (T^\lambda)=1$ and every weight $\mu$ of $T$ satisfies $\mu\preceq\lambda$.  Since $T$ has a generalized standard filtration and $\lambda$ is a maximal weight of $T$, by Proposition~\ref{prop:standardfiltration}(a), we know that $\Delta(\lambda)$ is a submodule of $T$ and $T/\Delta(\lambda)$ has a generalized standard filtration.  From Proposition~\ref{prop:tiltingext}, we know that $\Ext^1_\bbggO\big(T/\Delta(\lambda),D(\lambda)\big)=0$.
	
	Now from the short exact sequence $0\to \Delta(\lambda)\to T\to T/\Delta(\lambda)\to 0$ and from the long exact sequence of $\Ext$-groups, we have the following exact sequence
	\begin{align}
		\Hom_\bbggO\big(T/\Delta(\lambda),D(\lambda)\big)\to \Hom_\bbggO\big(T,D(\lambda)\big)\to \Hom_\bbggO\big(\Delta(\lambda),D(\lambda)\big)\to \Ext^1_\bbggO\big(T/\Delta(\lambda),D(\lambda)\big)\,.
	\end{align}
	Since $\Ext^1_\bbggO\big(T/\Delta(\lambda),D(\lambda)\big)=0$, the map $ \Hom_\bbggO\big(T,D(\lambda)\big)\to \Hom_\bbggO\big(\Delta(\lambda),D(\lambda)\big)$ is surjective.  Ergo, the embedding $\Delta(\lambda)\hookrightarrow D(\lambda)$ lifts to a homomorphism $\varphi:T\to D(\lambda)$.  
	
	Similarly, we also have a homomorphism $\psi:D(\lambda)\to T$ such that $\psi$ is an isomorphism on the copies of $\Delta(\lambda)$ in $D(\lambda)$ and $T$.  Thus, the endomorphism $\varphi\circ\psi:D(\lambda)\to D(\lambda)$ is an isomorphism on $\Delta(\lambda)\subseteq D(\lambda)$.  As $D(\lambda)$ is indecomposable, we know from \cite[Theorem 2.5]{DVN1} that every endomorphism of $D(\lambda)$ is either an isomorphism or a locally nilpotent map.  Since $\varphi\circ\psi$ preserves the weight space $\big(D(\lambda)\big)^\lambda$, the map $\varphi\circ\psi$ is not locally nipotent.  Hence, $\varphi\circ\psi$ is an isomorphism.  Consequently, both $\varphi$ and $\psi$ must be isomorphism, whence $T\cong D(\lambda)$.
\end{proof}

\begin{cor}
	If $T\in\bbggO$ is an indecomposable tilting module, then $T\cong D(\lambda)$ for some $\lambda\in\hfrak^*$.  In particular, all tilting modules are self-dual.
\end{cor}

\begin{proof}
	Let $\lambda$ be a maximal weight of $T$.  Using the same argument as the theorem above, we can easily see that $T\cong D(\lambda)$.  
	
	For the second part of the corollary, we let $T$ be an arbitrary tilting module.  We can then see from the paragraph above and \cite[Corollary 2.6]{DVN1} that $T=\bigoplus_\alpha D(\lambda_\alpha)$, where $\lambda_\alpha\in\hfrak^*$ for all $\alpha \in A$.  Since duality commutes with direct sum, it suffices to show that $D(\lambda)$ is self-dull for a fixed $\lambda\in\hfrak^*$.  As $D(\lambda)$ is an indecomposable tilting module, $\big(D(\lambda)\big)^\vee$ is also an indecomposable tilting module.  By the theorem above, we conclude that $\big(D(\lambda)\big)^\vee\cong D(\lambda)$.  
\end{proof}

\subsection{Multiplicities of Verma Factors in a Tilting Module}

   In this subsection, we shall again write $\Hom$ and $\Ext$ for $\Hom_\bbggO$ and $\Ext_\bbggO$.  We first need the following theorem.
      
   \begin{thm}
   Suppose that $M\in\boldsymbol{\Delta}(\bbggO)$.  For every $\lambda\in\hfrak^*$, we have
   \begin{align}
   	\big\{M,\Delta(\lambda)\big\}=\dim\Hom_\bbggO\big(M,\nabla(\lambda)\big)\,.
   \end{align}
   \label{thm:dimstandard}
   \end{thm}
   
   \begin{proof}
   	Without loss of generality, assume that $M$ is indecomposable.  We consider the set $\Pi_{\succeq \lambda}(M)$ of weights of $M$ that is greater than or equal to $\lambda$.  If this set is infinite, we can easily see that $M$ is not in the same block as $\Delta(\lambda)$.  Therefore, $\big\{M,\Delta(\lambda)\big\}=0$ and $\dim\Hom_\bbggO\big(M,\nabla(\lambda)\big)=0$.  Therefore, the assertion is true.  From now on, we assume that  $\Pi_{\succeq \lambda}(M)$ is finite.
   	
   	Define $m:=\sum_{\xi\succeq\lambda}\dim M^\xi$.  Then, $m$ is a nonnegative integer.  We can then perform induction on $m$, the base case $m=0$ being obvious.  Let now $m>0$.  Suppose that $\mu\succeq\lambda$ is a maximal weight of $M$.  By Proposition~\ref{prop:standardfiltration}(a), $M$ has a submodule $\Delta(\mu)$ such that $M/\Delta(\mu)$ has a generalized standard filtration.  From the short exact sequence $0\to \Delta(\mu)\to M\to M/\Delta(\mu)\to 0$ and the long exact sequence of $\Ext$-groups, we get the following exact sequence
   	\begin{align}
   		0\to \Hom\big(M/\Delta(\mu),\nabla(\lambda)\big)\to \Hom\big(M,\nabla(\lambda)\big)\to \Hom\big(\Delta(\mu),\nabla(\lambda)\big)\to \Ext^1\big(M/\Delta(\mu),\nabla(\lambda)\big)\,.
   	\end{align}
   	Because $M/\Delta(\mu)$ has a generalized standard filtration and $\nabla(\lambda)$ obviously has a generalized costandard filtration, Thorem~\ref{thm:extstandardcostandard} ensures that $\Ext^1\big(M/\Delta(\mu),\nabla(\lambda)\big)=0$.  Furthermore, because $\dim\Hom\big(\Delta(\mu),\nabla(\lambda)\big)=\deltaup_{\mu,\lambda}$, where $\deltaup$ is the Kronecker delta, we conclude that
   	\begin{align}
   		\dim\Hom\big(M,\nabla(\lambda)\big)=\dim\Hom\big(M/\Delta(\mu),\nabla(\lambda)\big)+\deltaup_{\mu,\lambda}\,.
   	\end{align}
   	On the other hand,
   	\begin{align}
   		\big\{M,\Delta(\lambda)\big\}=\big\{M/\Delta(\mu),\Delta(\lambda)\big\}+\big\{\Delta(\mu):\Delta(\lambda)\big\}=\big\{M/\Delta(\mu),\Delta(\lambda)\big\}+\deltaup_{\mu,\lambda}\,.
   	\end{align}
   	By induction hypothesis, $\big\{M/\Delta(\mu),\Delta(\lambda)\big\}=\dim\Hom\big(M/\Delta(\mu),\nabla(\lambda)\big)$, so $\dim\Hom\big(M,\nabla(\lambda)\big)$ and $\big\{M,\Delta(\lambda)\big\}$ are equal.
   \end{proof}
   
   \begin{cor}
   	   Suppose that $M\in\boldsymbol{\nabla}(\bbggO)$.  For every $\lambda\in\hfrak^*$, we have
   \begin{align}
   	\big\{M,\nabla(\lambda)\big\}=\dim\Hom_\bbggO\big(M^\vee,\nabla(\lambda)\big)\,.
   \end{align}
   \end{cor}
   
   Fix $\lambda\in\hfrak^*$.  For each positive integer $n$, we consider the restriction $\tilde{D}_n(\lambda):=\Res^{\gfrak}_{\gfrak_n}D(\lambda)$.  Because
   \begin{align}
   	\Res^{\gfrak}_{\gfrak_n}\Delta^{\gfrak}_{\bfrak}(\lambda)\cong\bigoplus_{\substack{\nu\preceq \lambda\\\lambda-\nu\notin\Lambda_{\gfrak_n,\hfrak_n}}}\,\Delta^{\gfrak_n}_{\bfrak_n}(\nu)
   	\label{eq:vermabranching}
   \end{align} 
   we can easily see that $\tilde{D}_n(\lambda)$ is a $\gfrak_n$-module with generalized standard filtration.  As the duality functor commutes with the restriction functor, we conclude that $\tilde{D}_n(\lambda)$ is a tilting $\gfrak_n$-module.
   
   Suppose that $\mu\in\hfrak^*$ satisfies $\mu\preceq\lambda$ and $\mu\in W_{\gfrak,\bfrak}[\lambda]\cdot\lambda$.  Let $n_0(\mu,\lambda)$ be the smallest positive integer $n$ such that $\lambda-\mu\in\Lambda_{\gfrak_n,\hfrak_n}$.   Due to Theorem~\ref{thm:goodfilter}(a), Equation \eqref{eq:vermabranching} implies that
   \begin{align}
   	\left\{D^\gfrak_\bfrak(\lambda):\Delta^\gfrak_\bfrak(\mu)\right\}=\left\{\tilde{D}_{n_0(\mu,\lambda)}(\lambda):\Delta^{\gfrak_{n_0(\mu,\lambda)}}_{\bfrak_{n_0(\mu,\lambda)}}(\mu)\right\}\,.
   \end{align}
   Since $D^{\gfrak_{n_0(\mu,\lambda)}}_{\bfrak_{n_0(\mu,\lambda)}}(\lambda)$ is the only indecomposable direct summand of $\tilde{D}_{n_0(\mu,\lambda)}(\lambda)$ that can contribute to the multiplicity of $\Delta^{\gfrak_{n_0(\mu,\lambda)}}_{\bfrak_{n_0(\mu,\lambda)}}(\mu)$ in $\tilde{D}_{n_0(\mu,\lambda)}(\lambda)$, we get
   \begin{align}
   	\left\{D^\gfrak_\bfrak(\lambda):\Delta^\gfrak_\bfrak(\mu)\right\}=\left\{D^{\gfrak_{n_0(\mu,\lambda)}}_{\bfrak_{n_0(\mu,\lambda)}}(\lambda):\Delta^{\gfrak_{n_0(\mu,\lambda)}}_{\bfrak_{n_0(\mu,\lambda)}}(\mu)\right\}\,.
   	\label{eq:reducedmultiplicity}
   \end{align}

  Recall from \cite[Theorem 3.8]{bggo} that, for each $\lambda\in\hfrak^*$ and for each positive integer $n$, $\bggO^{\gfrak_n}_{\bfrak_n}$ has enough projectives.  We let $\Plie^{\gfrak_n}_{\bfrak_n}(\lambda)$  denote the projective cover of the module $\Llie^{\gfrak_n}_{\bfrak_n}(\lambda)$ in $\bggO^{\gfrak_n}_{\bfrak_n}$.
  
 \begin{thm}
   	Let $\lambda,\mu\in\hfrak^*$ with $\mu\preceq\lambda$ and $\mu\in W_{\gfrak,\bfrak}[\lambda]\cdot\lambda$.  Write $n:=n_0(\mu,\lambda)$.  Fix $\xi\in\hfrak^*$ such that $\xi$ is a $\bfrak_{n}$-antidominant weight in $W_{\gfrak_{n},\bfrak_n}[\lambda]\cdot\lambda$.  If $w^0_n$ is the longest element of $W_{\gfrak_{n},\hfrak}$, then
   	\begin{align}
   		\left\{D^\gfrak_\bfrak(\lambda):\nabla^\gfrak_\bfrak(\mu)\right\}=\left\{D^\gfrak_\bfrak(\lambda):\Delta^\gfrak_\bfrak(\mu)\right\}=\left\{\Plie^{\gfrak_{n}}_{\bfrak_{n}}\big(w^0_{n}\cdot\lambda\big),\Delta^{\gfrak_n}_{\bfrak_n}\big(w^0_{n}\cdot \mu\big)\right\}
   	\end{align}
 \end{thm}
 
 \begin{proof}
     Due to \cite[Theorem 6.10]{BeilinsonGinzburg}, we have $\big\{D^{\gfrak_n}_{\bfrak_n}(\lambda),\Delta^{\gfrak_n}_{\bfrak_n}(\mu)\big\}=\left\{\Plie^{\gfrak_{n}}_{\bfrak_{n}}\big(w^0_{n}\cdot\lambda\big),\Delta^{\gfrak_n}_{\bfrak_n}\big(w^0_{n}\cdot \mu\big)\right\}$.  The theorem follows immediately from \eqref{eq:reducedmultiplicity}.
 \end{proof}
   
   Write $P^W_{x,y}(q)\in\mathbb{Z}[q]$ for the Kazhdan-Lusztig polynomial for elements $x,y$ in a Coxeter group $W$.  Due to \eqref{eq:reducedmultiplicity}, we may assume without loss of generality that $\lambda$ and $\mu$ are integral weights of $\gfrak_{n_0(\mu,\lambda)}$.  We then have the following theorem.
   
   \begin{thm}
   	Let $\lambda,\mu\in\hfrak^*$ with $\mu\preceq\lambda$ and $\mu\in W_{\gfrak,\bfrak}[\lambda]\cdot\lambda$.  Suppose that $\lambda$ is a regular integral weight with respect to $\gfrak_{n_0(\mu,\lambda)}$.   Fix $\xi\in\hfrak^*$ such that $\xi$ is a $\bfrak_{n_0(\mu,\lambda)}$-antidominant weight in $W_{\gfrak_{n_0(\mu,\lambda)},\bfrak_{n_0(\mu,\lambda)}}[\lambda]\cdot\lambda$.  If $\lambda=x\cdot\xi$ and $\mu=y\cdot \xi$ for some $x,y\in W_{\gfrak_{n_0(\mu,\lambda)},\hfrak}$, then
   	\begin{align}
   		\left\{D^\gfrak_\bfrak(\lambda):\nabla^\gfrak_\bfrak(\mu)\right\}=\left\{D^\gfrak_\bfrak(\lambda):\Delta^\gfrak_\bfrak(\mu)\right\}=P^{W_{\gfrak_{n_0(\mu,\lambda)},\hfrak}}_{y,x}(1)\,.
	\end{align}
   \end{thm}
   
  \begin{proof}
  	 For simplicity, write $n:=n_0(\mu,\lambda)$.  From \cite[Theorem 4.4]{Soergel2008}, we have
  	 \begin{align}
  	 	\dim\Hom_{\bggO^{\gfrak_n}_{\bfrak_n}}\left(\Delta^{\gfrak_n}_{\bfrak_n}(\mu),D^{\gfrak_n}_{\bfrak_n}(\lambda)\right)=\dim\Hom_{\bggO^{\gfrak_n}_{\bfrak_n}}\left(\Delta^{\gfrak_n}_{\bfrak_n}(y\cdot\xi),D^{\gfrak_n}_{\bfrak_n}(x\cdot \xi)\right)=P^{W_{\gfrak_n,\hfrak}}_{y,x}(1)\,.
  	 \end{align}
  	 From Theorem~\ref{thm:dimstandard}, we have $\dim\Hom_{\bggO^{\gfrak_n}_{\bfrak_n}}\left(\Delta^{\gfrak_n}_{\bfrak_n}(\mu),D^{\gfrak_n}_{\bfrak_n}(\lambda)\right)=\big\{D^{\gfrak_n}_{\bfrak_n}(\lambda),\Delta^{\gfrak_n}_{\bfrak_n}(\mu)\big\}$.  By \eqref{eq:reducedmultiplicity}, the claim follows immediately.
  \end{proof}






	
	
				\Addresses

\end{document}